\documentclass[11p,leqno]{amsart}
\usepackage[english]{babel}
\usepackage{amsmath,amsthm}
\usepackage{amsfonts}
\usepackage{amssymb}
\usepackage{graphicx}
\usepackage{color}
\usepackage{amsfonts,mathtools,mathrsfs}
\usepackage{amscd}
\usepackage{esint}

\newtheorem{theorem}{Theorem}[section]%\newtheorem{thm}[theorem]{Theorem}
\newtheorem*{theorem*}{Theorem}
\newtheorem{lemma}{Lemma}[section]
\newtheorem{corollary}{Corollary}[section]
\newtheorem{proposition}{Proposition}[section]

\newtheorem{definition}{Definition}[section]

\newtheorem{remark}{Remark}[section]

\numberwithin{equation}{section}

\begin{document}

\title[Suitable weak solutions of NSPNP]{Partial regularity of suitable weak solutions of the Navier-Stokes-Planck-Nernst-Poisson equation}
\author{Huajun Gong, Changyou Wang, Xiaotao Zhang}
\address{Shenzhen Key Laboratory of Advance Machine Learning and Applications,
College of Mathematics and Statistics, Shenzhen University, Shenzhen 518060, Guangdong, China}
\email{huajun84@szu.edu.cn}

\address{Department of Mathematics, Purdue University, West Lafayette, IN, 47907, USA.} 
\email{wang2482@purdue.edu}

\address{South China Research Center for Applied Mathematics and Interdisciplinary
Studies, South \ China Normal University, Zhong Shan Avenue West 55, Guangzhou
510631, China}
\email{xtzhang@m.scnu.edu.cn}

\maketitle

\begin{abstract}
In this paper, inspired by the seminal work by Caffarelli-Kohn-Nirenberg \cite{CKN} on the incompressible Navier-Stokes equation,
we establish the existence of a suitable weak solution to the Navier-Stokes-Planck-Nernst-Poisson 
equation in dimension three, which is shown to be smooth away from a closed set whose $1$-dimensional parabolic Hausdorff measure is zero.
\end{abstract}
\maketitle

\section{Introduction}
Let $\Omega \subset \mathbb{R}^3$ be  a bounded, smooth domain and $0<T< \infty$. 
We consider the following Navier-Stokes-Nernst-Planck-Poisson equation:
\begin{equation}\label{B}\left
\{\begin{array}{l}
\large{ \partial _t u +(u\cdot \nabla ) u -\Delta u + \nabla P = -(n^+-n^-)\nabla \Psi,}\\
\large{\text{div}\ u=0,}\\
\large{ \partial_t n^+ + (u\cdot \nabla )n^+ -\Delta n^+ = \text{div} (n^+ \nabla \Psi),}\\
\large{ \partial_t n^- + (u\cdot \nabla )n^- -\Delta n^- = -\text{div} (n^- \nabla \Psi),}\\
\large{ -\Delta \Psi =n^+ -n^- ,}
\end{array}
\right. \ \ {\rm{in}}\ \  \Omega\times (0,T),
\end{equation}
where $u:\Omega\times(0,T) \rightarrow \mathbb{R}^3$ denotes the velocity field of fluid, 
$P:\Omega\times(0,T)\rightarrow \mathbb{R}$ denotes the pressure function, $n^+, n^-:\Omega\times(0,T) \rightarrow \mathbb{R}$ are the number densities of positively and negatively charged constituents, 
and $\Psi$ is the quasi-electrostatic potential field. Along with (\ref{B}), the initial and boundary values are:
\begin{equation}\label{initial}
  (u,n^+,n^-)=(u_0,n^+_0,n^-_0), \ \  \text{in  } \  \ \Omega\times\{0\},
\end{equation}
\begin{equation}\label{boundary}
  u=0,\quad \frac{\partial {n^+}}{\partial\nu}=\frac{\partial {n^-}}{\partial\nu}=\frac{\partial\Psi}{\partial\nu}=0 , \quad \text{on  } \partial\Omega\times (0,T),
\end{equation}
where $\nu$ denotes the exterior unit normal vector on $\partial \Omega$.

The system \eqref{B} models an isothermal, incompressible, viscous Newtonian fluid of uniform and homogeneous composition of a high number of positively and negatively charged particles ranging from colloidal to nano size. It was proposed by Rubinstein \cite{Ru} to model electro-kinetic fluids, which describes the interaction between the macroscopic fluid motion and the microscopic charge transportion. 
See Castellanos \cite{AC} for more discussions on the physics associated with \eqref{B}. 
In the system \eqref{B}, we assume a dilute fluid and therefore the electromagnetic forces are neglected. There have seen considerable interests in the mathematical analysis of the system \eqref{B}.
For example,  Jerome \cite{JJ}  has proved the existence of local strong solutions under the Kato's semigroup framework. Deng-Zhao-Cui \cite{DZC} have
established the existence and well-posedness of mild solutions in the Triebel-Lizorkin and Besov spaces of negative indices. 
We refer to Zhao-Zhang-Liu \cite{ZZL} for some time decay results of \eqref{B}. The existence of global weak solutions
of \eqref{B}, \eqref{initial} and \eqref{boundary} has been established by Schmuck \cite{Sch} under the Neumann boundary condition (for
initial data with bounded $n_0^+$ and $n_0^-$),
and Jerome-Sacco \cite{JS}, under the mixed Dirichlet boundary condition. 
Fan-Li-Nakamura \cite{FLN} have proved some regularity criteria of weak solutions to \eqref{B} on $\Omega=\mathbb R^3$ in the spirit
of Serrin. More recently, there are some interesting works by Wang-Liu-Tan \cite{WLT1, WLT2} on generalized Navier-Stokes-Planck-Nernst-Poisson equations through an energetic-variational approach.

When the underlying fluid is at rest $u=0$, the system \eqref{B} reduces to the Planck-Nernst-Poisson (PNP) equation, 
which is the drift-diffusion model for semiconductor devices, first proposed by Roosbroeck \cite{Roo} in 1950, 
that has been widely accepted and applied in semiconductor industry and in device simulation.  See Gajewski \cite{G},
Mock \cite{M}, Seidman-Troianiello \cite{ST}, and Fang-Ito \cite{FI} for results on the existence of
global weak solutions to the PNP equation.

 It remains to be an interesting question to investigate regularity properties
of weak solutions in three dimension. Motivated by the celebrated work by Scheffer \cite{Sche}, Caffarelli-Kohn-Nirenberg \cite{CKN}, and Lin \cite{LIN} on the Navier-Stokes equation, 
we introduce the notion of suitable weak solution of (\ref{B})-\eqref{initial}-(\ref{boundary}) and establish both the existence and 
partial regularity for such a weak solution. See also \cite{Hynd}, \cite{LL}, and \cite{DHW} for related works on other complex fluids.

A constitutive equation of the Navier-Stokes-Nernst-Planck-Poisson system \eqref{B} 
is the Naiver-Stokes equation: for $0<T\le \infty$, 
\begin{equation}\label{n-s}
\left\{
  \begin{array}{ll}
      \partial_t u+(u\cdot\nabla)u -\nu \Delta u + \nabla P= f, \\
  \nabla \cdot  u = 0,
  \end{array}
\right. \ \ {\rm{in}}\ \ \ Q_T=\Omega\times (0,T),
\end{equation}
with the initial-boundary condition
\begin{equation}\label{n-s-initial}
  u(\cdot,0)=u_0 \ {\rm{in}}\ \Omega; \quad  u=0 \ {\rm{on}}\ {\partial\Omega}\times [0,T).
\end{equation}
The existence of global weak solutions of \eqref{n-s} and \eqref{n-s-initial} ($T=\infty$) was established by
Leray \cite{Leray} and Hopf \cite{Temam}. While it is an outstanding open question whether \eqref{n-s} and \eqref{n-s-initial}
has a global smooth solution when $\Omega=\mathbb R^3$, there has been many research works concerning partial regularity of suitable weak solutions of \eqref{n-s}
initiated by Scheffer \cite{Sche} and then by Caffarelli-Kohn-Nirenberg \cite{CKN}, where it was proven
that the singular set has $1$-dimensional Hausdorff measure zero. Such a theorem was later simplified by Lin \cite{LIN}. 
There has also been a lot of work on the regularity criteria of \eqref{n-s}
going back to Serrin \cite{Serrin} where it has been proven that $u\in C^\infty(Q_T)$, provided $u\in L^q_tL^p_x(Q_T)$, where $p\ge 3$ and $2\le q<\infty$
satisfy
\begin{equation}\label{serrin}
\frac{2}{q}+\frac{3}{p}=1.
\end{equation}  
The end point case $p=3$ and $q=\infty$ for \eqref{serrin} was resolved by \cite{SS}.

The goal of this paper is to extend the partial regularity theory on the Navier-Stokes equation by Caffarelli-Kohn-Nirenberg \cite{CKN}
to the system \eqref{B}. We first recall the definition of suitable weak solutions to the system (\ref{B}). For $T>0$, denote
$Q_T=\Omega \times (0,T)$ and
$$
\mathcal{D}:= \Big\{ X \ | \ X\in C_0^{\infty}(Q_T, \mathbb{R}^3), \ \text{div} X=0 \Big\}.
$$
\begin{definition}\label{Def1} We say that 
$(u, n^+, n^-, \Psi)$ is a weak solution of (\ref{B}) in $Q_T$, if
$$u\in L^{\infty}([0,T], L^2(\Omega,\mathbb R^3)) \cap L^2([0,T], H^1(\Omega,\mathbb R^3)),$$
$$\Psi\in  L^{\infty}([0,T], H^1(\Omega)) \cap L^2([0,T], H^2(\Omega)), \ \ n^+,n^-\in L^2 (Q_T,\mathbb R_+),$$
and the system (\ref{B})  holds in the sense of distributions:
 for any $\varphi \in \mathcal{D}$,\\
$$
\int_{Q_T}\big( \langle u, \partial_t\varphi \rangle-\langle \nabla u, \nabla\varphi \rangle+u\otimes u: \nabla \varphi\big)\ dxdt
=\int_{Q_T}\langle (n^+- n^-)\nabla \Psi,   \varphi \rangle \ dxdt,
$$
and, for any $\phi \in C_0^{\infty}(Q_T)$,
$$
\int_{Q_T} u\cdot\nabla\phi\ dxdt=0,
$$

$$
\int_{Q_T} \big(\langle n^+, \partial_t\phi\rangle+\langle n^+, \Delta \phi \rangle+\langle n^+, u
\cdot \nabla \phi \rangle\big)\ dxdt
=\int_{Q_T} \langle n^+\nabla \Psi,  \nabla \phi \rangle \ dxdt,
$$
$$
\int_{Q_T} \big(\langle n^-, \partial_t\phi\rangle+\langle n^-, \Delta \phi \rangle+\langle n^-, u
\cdot \nabla \phi \rangle\big)\ dxdt
=-\int_{Q_T} \langle n^-\nabla \Psi,  \nabla \phi \rangle \ dxdt,
$$
and
$$\int_{\Omega} \langle\nabla\Psi, \nabla \phi\rangle\,dx=\int_{\Omega}(n^+-n^-)\phi\,dx, \ \ \forall 0<t<T, $$
where $\langle \cdot,\cdot \rangle$ denotes the inner product of $\mathbb{R}^3$.
\end{definition}

A weak solution $(u, P, n^+, n^-, \Psi)$ is called a suitable weak solution of \eqref{B}, if, in addition, it enjoys 
the following properties.

\begin{definition}\label{Def2}
A weak solution $(u,P,n^+,n^-, \Psi)$ of \eqref{B} is called a suitable weak solution of \eqref{B} in $Q_T$, if the following conditions are true:
\begin{enumerate}
  \item [(a)] $P\in L^{\frac{3}{2}}(Q_T)$, 
  \item [(b)] $n^+, n^- \in  L^2(Q_T)$, 
  \item[(c)] there exist positive constants $0<E_1, E_2< \infty$ such that,
$$\begin{cases}\displaystyle
\int_{\Omega} (|u|^2 + |\nabla \Psi|^2(x,t) \,dx \le E_1, \ \ \forall t\in (0,T),\\
\displaystyle\int_{Q_T} (|\nabla u|^2 + |\nabla^2 \Psi|^2)\, dxdt \le E_2,
\end{cases}
$$
\item $(u,P,n^+,n^-, \Psi)$ satisfy \eqref{B} in the sense of distributions on $Q_T$.
\item for any $\phi\in C^\infty(Q_T), \phi\geq 0$, the generalized energy inequality (\ref{GE-1}) holds:
\begin{equation}\label{GE-1}
 \begin{split}
     2\int _{Q_T}|\nabla u|^2 \phi \,dxdt 
&\leq  \int_{Q_T}|u|^2(\partial_t\phi +\Delta \phi) \,dxdt + \int_{Q_T}(|u|^2 +2P)u\cdot \nabla \phi \,dxdt\\
    &\ \ \ -2\int_{Q_T} \big(\nabla\Psi \otimes\nabla\Psi-\frac12|\nabla\Psi|^2I_3\big): \nabla (u\phi) \,dxdt.
 \end{split}
\end{equation}
\end{enumerate}

\end{definition}
%Motivated by the results of \cite{MR673830, MR1367385} and the relation between the system (\ref{B})-(\ref{boundary}) and the Navier-Stokes equations (\ref{n-s})-(\ref{n-s-intial}), we will show the following partial results for system (\ref{A1})-(\ref{boundary0}).
Now we are ready to state our main theorem.

\begin{theorem}\label{main} For any $0<T\le\infty$, $u_0\in L^2(\Omega,\mathbb R^3)$, with ${\rm{div}}u_0=0$,
and $0\le n^+_0, n^-_0\in L^2(\Omega)$, with $\displaystyle\int_\Omega n_0^+\,dx= \int_\Omega n_0^-\,dx,$
there exists a suitable
weak solution $(u, P, n^+, n^-,\Psi)$ of \eqref{B}-\eqref{initial}-\eqref{boundary} in $Q_T$ such that
\begin{enumerate}
\item [(i)]
$u\in L^\infty_tL^2_x\cap L^2_tH^1_x(Q_T)$, $P\in L^\frac53(Q_T)$, $0\le n^+,n^-\in L^\infty_tL^2_x\cap L^2_tH^1_x(Q_T)$, $\Psi\in L^\infty_tH^2_x\cap L^2_tH^3_x(Q_T)$, and
\begin{equation}\label{a-bound}
\begin{split}
&\big\|(u, n^+, n^-)\big\|_{L^\infty_tL^2_x\cap L^2_tH^1_x(Q_T)} +\big\|P\big\|_{L^\frac53(Q_T)}
+\big\|\Psi\big\|_{L^\infty_tH^2_x\cap L^2_tH^3_x(Q_T)}\\
&\le C\big(\|u_0\|_{L^2(\Omega)}, \|n^+_0\|_{L^2(\Omega)},  \|n^-_0\|_{L^2(\Omega)}\big).
\end{split}
\end{equation}
\item [(ii)] $(u, n^+, n^-, \Psi)$ satisfies the following global energy inequality: for any $0<t\le T$,
\begin{equation}\label{Global-EI}
\begin{split}
&\int_\Omega (|u|^2+|\nabla\Psi|^2)(x,t)\,dx+2\int_{Q_t}(|\nabla u|^2+|n^+-n^-|^2+(n^++n^-)|\nabla\Psi|^2)\,dxds\\
&\le \int_{\Omega}(|u_0|^2+|\nabla\Psi_0|^2)(x)\,dx,
\end{split}
\end{equation} 
where $\Psi_0\in H^2(\Omega)$ solves 
$$-\Delta\Psi_0=n_0^+-n_0^- \ {\rm{in}}\ \Omega; \ \frac{\partial\Psi_0}{\partial\nu}=0 \ {\rm{on}}\ \partial\Omega.$$
\item [(iii)] there exists a closed set $\Sigma\subset Q_T$, with $\mathcal{P}^1(\Sigma)=0$, such that
$(u, n^+, n^-,\Psi)\in C^\infty(Q_T\setminus\Sigma).$
\end{enumerate}
\end{theorem}

Here $\mathcal{P}^k$, $0\le k\le 4$, denotes the $k$-dimensional Hausdorff measure on $\mathbb R^4$ with respect to the parabolic distance:
$$\delta((x,t), (y,s))=\max\big\{|x-y|,\sqrt{|t-s|}\big\}.$$

We would like to briefly mention some key steps of the proof of Theorem \ref{main}:
\begin{enumerate}
\item The existence of suitable weak solutions to \eqref{B} is established by first studying approximate systems
of \eqref{B} through modifying an ``retarded" mollification of its drifting coefficients, $\Theta_\epsilon(u)$, originally due to \cite{CKN} on the Navier-Stokes equation. 
Here we need to modify it so that its normal component vanishes on the boundary of $\Omega$ in order to guarantee the equations for $n^\pm$ enjoy both
the positivity and maximum principle property.
For the existence of suitable weak solutions to an approximate version of \eqref{B}, we employ a contraction map theorem on the
function spaces $L^4_tL^2_x(Q_T)$ for $n^\pm$ first employed by  Schmuck \cite{Sch}. Then we prove that such a sequence of suitable weak solutions to the approximate equation enjoy
some uniform estimates and hence converge to a suitable
weak solution to \eqref{B}. 
\item The partial regularity of a suitable weak solution constructed in (1) is proven by employing the fact $\Psi\in L^\infty_tH^2_x(Q_T)$ 
to perform a blowing up argument to establish an $\epsilon_0$-decay property for $(u,P)$ in 
the renormalized $L^3\times L^\frac32$-norms, and then apply
the Reisz potential estimates on parabolic Morrey spaces to obtain $L^q$-boundedness of $u$ for any $1<q<\infty$, which can yield
the $\epsilon_0$-smoothness of $(u, n^+, n^-, \Psi)$ via the bootstrap argument.
\item To obtain the size estimate of the singular set, we improve the $\epsilon_0$-regularity from (2) in a way similar to that of the Navier-Stokes equation by \cite{CKN} through establishing the so-called the ABCD Lemmas. 
\end{enumerate}

The paper is organized as follows. In section 2, we will establish the existence of the suitable weak solutions of (\ref{B})-\eqref{initial}-(\ref{boundary}).
In section 3, we will prove an $\epsilon_0$-regularity for suitable weak solutions to \eqref{B}. In section 2, we will improve the $\epsilon_0$-regularity
from section 3 and provide a proof of Theorem \ref{main}.

\section{Existence of suitable weak solutions}

In order to obtain the existence of suitable weak solutions of  \eqref{B}, we first consider the following system:
given $w\in C^{\infty}(\overline{\Omega}\times [0, T], \mathbb{R}^3)$ with div $w=0$ in $Q_T$ and
$w\cdot\nu=0$ on $\partial\Omega\times [0,T]$ , let $(u,P, n^+, n^-,\Psi)$
solve
\begin{equation}\label{A2}\left
\{\begin{array}{l}
\large{ \partial _t u +(w\cdot \nabla ) u -\Delta u + \nabla P = -(n^+-n^-)\nabla\Psi,}\\
\large{\text{div} u=0,}\\
\large{ \partial_t n^+ + (w\cdot \nabla )n^+ -\Delta n^+ = \text{div}([n^+]_+ \nabla \Psi),}\\
\large{ \partial_t n^- + (w\cdot \nabla )n^- -\Delta n^- = -\text{div}([n^-]_+ \nabla \Psi),}\\
\large{ -\Delta \Psi =n^+ -n^- ,}
\end{array}
\right.
\end{equation}
subject to the initial and boundary condition:
\begin{equation}\label{IC}
(u, n^+, n^-)|_{t=0}=(u_0, n^+_0, n^-_0) \ \ \text{in  }\ \Omega,
\end{equation} 
\begin{equation}\label{BC}
u=0, \ \ \frac{\partial n^+}{\partial\nu}=\frac{\partial {n^-}}{\partial\nu}=\frac{\partial\Psi}{\partial\nu}=0
\ \ \text{on  } \partial \Omega\times (0,T).
\end{equation}
Here $[y]_+=\max\{y, 0\}$ denotes the positive part of $y\in\mathbb R$.

We shall use the following function spaces:
\begin{equation*}
  \mathcal{V} =C_0^\infty(\Omega,\mathbb{R}^3) \cap \{ u: {\rm{div}} u=0 \},
\end{equation*}
\begin{equation*}
  {\bf H}=\text{Closure of } \mathcal{V} \text{ in } L^2(\Omega),
\end{equation*}
\begin{equation*}
  {\bf V}=\text{Closure of } \mathcal{V} \text{ in } H^1(\Omega).
\end{equation*}

Concerning \eqref{A2}, \eqref{IC} and \eqref{BC}, we have the following existence result.

\begin{theorem}\label{linear-A1}
For a bounded and smooth domain $\Omega\subset\mathbb R^3$, $u_0\in {\bf H}$,
and two nonnegative $n^+_0, n^-_0\in L^2(\Omega)$ satisfying 
$$\int_\Omega n_0^+(x)\,dx=\int_\Omega n_0^-(x)\,dx,$$
if  $w\in C^{\infty}(\overline{\Omega}\times [0, T], \mathbb{R}^3)$, with div $w=0$ in $Q_T$ and $w\cdot\nu=0$ on $\partial\Omega\times [0,T]$,
then there is a unique weak solution $(u, P, n^+,n^-,\Psi)$ of \eqref{A2}, \eqref{IC}, and \eqref{BC}
such that $n^+, n^-\ge 0$ in $\Omega\times [0,T]$, and
\begin{equation}\label{weak_sol1}
\begin{cases} u\in C([0,T], {\bf H}) \cap L^2([0,T], {\bf V}),\\
\Psi\in L^\infty([0,T], H^2(\Omega)) \cap L^2([0, T], H^3(\Omega)),\\
n^+, n^- \in L^\infty([0,T], L^2(\Omega))\cap L^2([0,T], H^1(\Omega)).
\end{cases}
\end{equation}
\end{theorem}

The existence of weak solutions $(u, P, n^+, n^-, \Psi)$ to \eqref{A2}, \eqref{IC}, and \eqref{BC} will be established
by a contraction map argument. The uniqueness of such weak solutions $(u, P, n^+, n^-, \Psi)$ can be employed 
to show the non-negativity of $n^+, n^-$ as follows.

\begin{lemma}\label{non-negative} Under the assumptions of Theorem \ref {linear-A1},
the weak solution $(u, P, n^+, n^-,\Psi)$ of \eqref{A2}, \eqref{IC} and \eqref{BC}, satisfying \eqref{weak_sol1}, 
must satisfy $n^+, n^-\ge 0$ in $Q_T$.
\end{lemma}
\begin{proof} This proof is similar to Lemma 1 in \cite{Sch}.
In order to prove that $n^+,n^-$ are non-negative, let $(\tilde{u}, \tilde{P}, \tilde{n}^+, \tilde{n}^-, \tilde{\Psi})$, satisfying
\eqref{weak_sol1},  be a weak solution of the system:
\begin{equation}\label{A3}
\left
\{\begin{array}{l}
\large{ \partial _t \tilde{u} +(w\cdot \nabla ) \tilde{u} -\Delta \tilde{u} + \nabla \tilde{P} = -(\tilde{n}^+-\tilde{n}^-)\nabla\tilde\Psi,}\\
\large{\text{div} \tilde{u}=0,}\\
\large{ \partial_t \tilde{n}^+ + (w\cdot \nabla )\tilde{n}^+ -\Delta \tilde{n}^+ = \text{div}([\tilde{n}^+]_+ \nabla \tilde\Psi),}\\
\large{ \partial_t \tilde{n}^- + (w\cdot \nabla )\tilde{n}^- -\Delta \tilde{n}^- = -\text{div}([\tilde{n}^-]_+ \nabla \tilde\Psi),}\\
\large{ -\Delta \tilde\Psi =\tilde{n}^+ -\tilde{n}^- ,}
\end{array}
\right.\ {\rm{in}}\  Q_T, 
\end{equation}
subject to the initial and boundary condition:
\begin{equation}\label{IC1}
(\tilde{u}, \tilde{n}^+, \tilde{n}^-)|_{t=0}=(u_0, n^+_0, n^-_0) \ \ \text{in  }\ \Omega,
\end{equation} 
\begin{equation}\label{BC1}
\tilde{u}=0, \ \ \frac{\partial \tilde{n}^+}{\partial\nu}=\frac{\partial {\tilde{n}^-}}{\partial\nu}=\frac{\partial\tilde{\Psi}}{\partial\nu}=0
\ \ \text{on  } \partial \Omega\times (0,T).
\end{equation}
The existence of such. a weak solution
$(\tilde{u}, \tilde{P}, \tilde{n}^+, \tilde{n}^-, \tilde\Psi)$ will be constructed
by Theorem 2.1 below. 

It is readily seen that $\tilde{n}^+=[\tilde{n}^+]_+-[-\tilde{n}^+]_+$.
Multiplying \eqref{A3}$_3$ by $[-\tilde{n}^+]_+$ and integrating over $\Omega$, we have that
$$\frac{1}{2} \frac{d}{dt}\int_{\Omega} |[-\tilde{n}^+]_+|^2 \,dx 
+\int_{\Omega} |\nabla [-\tilde{n}^+]_+|^2\,dx 
=\int_{\Omega} [\tilde{n}^+]_+ \nabla [-\tilde{n}^+]_+\cdot \nabla \tilde\Psi \,dx=0$$
This implies that 
$$\int_{\Omega} |[-\tilde{n}^+]_+|^2\,dx\leq \int_{\Omega} |[-n_0]_+|^2 \,dx=0,$$
since $n^+_0$ is non-negative.  Thus we conclude that
$\tilde{n}^+\ge 0$ in $Q_T$. Similarly, we can show that $\tilde{n}^-\ge 0$ in $Q_T$.
Therefore, we see that $(\tilde{u}, \tilde{P}, \tilde{n}^+, \tilde{n}^-, \tilde{\Psi})$ is also a weak solution of
\eqref{A2}, \eqref{IC}, and \eqref{BC}. From Theorem 2.1, 
 the uniqueness holds for weak solutions to \eqref{A2}, \eqref{IC}, and \eqref{BC}, satisfying \eqref{weak_sol1}.
Thus
$$(\tilde{u}, \tilde{P}, \tilde{n}^+, \tilde{n}^-, \tilde{\Psi})\equiv(u, P, n^+, n^-,\Psi)  \ {\rm{in}}\ Q_T.$$
Hence $n^+\equiv\tilde{n}^+\ge 0$ and $n^-\equiv\tilde{n}^-\ge 0$ in $Q_T$. \qed

\begin{proposition}\label{bound} Under the same assumptions as Theorem \ref {linear-A1},
if, in addition, $n^+_0, n^-_0\in L^p(\Omega)$ for some $p\ge 2$, then the weak solution $(u, P, n^+, n^-,\Psi)$ of \eqref{A2}, \eqref{IC} and \eqref{BC},
satisfying \eqref{weak_sol1},  enjoys
\begin{equation}\label{weak_sol2}
n^+, n^-\in L^\infty([0,T], L^p(\Omega)), \ \Psi\in L^\infty([0, T], W^{2,p}(\Omega)),
\end{equation}
and
\begin{equation}\label{lp-est}
\begin{split}
&\int_\Omega (|n^+|^p+|n^-|^p)(x, t)\,dx
+p(p-1)\int_{Q_t} \big[(n^+)^{p-2}|\nabla n^+|^2+(n^-)^{p-2}|\nabla n^-|^2\big]\,dxdt\\
&\le \int_\Omega (|n^+_0|^p+|n^-_0|^p)(x)\,dx, \ 0\le t<T.
\end{split}
\end{equation}
\end{proposition}
\begin{proof}
Multiplying $(\ref{A2})_3$ by $|n^+|^{p-2} n^+$ and $(\ref{A2})_4$ by  $|n^-|^{p-2}n^-$,
integrating the resulting equations over $\Omega$, and applying \eqref{A2}$_5$, we obtain that 
\begin{equation*}
  \begin{split}
     & \frac{1}{p}\frac{d}{dt}\int_{\Omega}(|n^+|^{p}+|n^-|^{p})\,dx \\
    &+(p-1)\int_{\Omega}(|\nabla n^+|^2|n^+|^{p-2} +|\nabla n^-|^2|n^-|^{p-2})\,dx\\
    &=-\frac{p-1}{p}\int_\Omega \nabla\Psi\cdot \nabla (|n^+|^p-|n^-|^p)\,dx\\
    &=-\frac{p-1}{p}\int_\Omega  (|n^+|^p-|n^-|^p)(n^+-n^-)\,dx \le 0,
  \end{split}
\end{equation*}
where we have used in the last step the fact that $n^+,n^- $ are non-negative, and
the inequality
$$(|n^+|^p-|n^-|^p)(n^+-n^-)=[(n^+)^p-(n^-)^p] [n^+-n^-]\ge 0.$$
Therefore we obtain that
$$\frac{d}{dt}\int_{\Omega}(|n^+|^{p}+|n^-|^{p})\,dx+(p-1)\int_{Q_t} \big[(n^+)^{p-2}|\nabla n^+|^2+(n^-)^{p-2}|\nabla n^-|^2\big]\,dxdt\le 0.$$
This implies \eqref{lp-est} and completes the proof.
\end{proof}

\begin{proof}[Proof of Theorem \ref{linear-A1}] {\it Step 1}: Existence.
We will modify the approach by Schmuck \cite{Sch}. For $T>0$, set the function space
$$Y_T\equiv\big\{ {\bf y}=(n^+, n^-): n^\pm \in L^4([0,T], L^2(\Omega))\big\},$$
equipped with the norm
$$\big\|(n^+,n^-)\big\|_{Y_T}=\big\|(n^+, n^-)\big\|_{L^4([0,T], L^2(\Omega))}.$$
Now we define a map $F:Y_T\mapsto Y_T$ as follows. For any $\overline{\bf y}=(\bar{n}^+, \bar{n}^-)$,
define $F(\overline{\bf y})={\bf y}=(n^+, n^-)$, where ${\bf y}$ is a solution of the system:
\begin{equation}\label{psi-eqn}
-\Delta \overline{\Psi}={\bar n}^+-{\bar n}^- \ {\rm{in}}\ \Omega; \ \  \frac{\partial \overline{\Psi}}{\partial\nu}=0 
\ {\rm{on}}\ \partial\Omega,
\end{equation}
\begin{equation}\label{y-eqn0}
\begin{cases}
\partial_t n^+ +(w\cdot\nabla) n^+-\Delta n^+={\rm{div}}([n^+]_+\nabla\overline{\Psi}) & {\rm{in}}\ Q_T,\\
\partial_t n^- +(w\cdot\nabla) n^--\Delta n^-=-{\rm{div}}([n^-]_{-}\nabla\overline{\Psi}) & {\rm{in}}\ Q_T,\\
(n^+, n^-)=(n^+_0, n^-_0) & {\rm{on}}\ \Omega\times\{t=0\}, \\
\displaystyle\frac{\partial n^+}{\partial\nu}= \frac{\partial n^-}{\partial\nu}=0 & {\rm{on}}\ \partial\Omega\times [0,T].
\end{cases}
\end{equation}
Note that for any $f, g\in L^1(\Omega)$, it holds that
$$|[f]_+|\le |f|, \ |[f]_+ - [g]_+|\le |f-g| \ {\rm{a.e.}}\ \Omega.$$
Since ${\bar n}^+-{\bar n}^-\in L^4([0,T], L^2(\Omega))$, it follows from $W^{2,2}$-theory of the
Laplace equation that $\overline{\Psi}\in L^4([0,T], W^{2,2}(\Omega))$, and
\begin{equation}\label{22-est}
\big\|\overline{\Psi}\big\|_{L^4([0,T], W^{2,2}(\Omega))}
\leq C \big\|{\bar n}^+-{\bar n}^-\big\|_{L^4([0,T], L^2(\Omega))}\leq C\big\|\bar{\bf y}\big\|_{Y_T}.
\end{equation}
By the theory of linear parabolic systems \cite{LSN},  
there exists a unique solution $(n^+, n^-)$ of \eqref{y-eqn0} in $L^\infty([0,T], L^2(\Omega))\cap L^2([0,T], H^1(\Omega))$ for any $T>0$.  
Moreover, by multiplying \eqref{y-eqn0}$_1$ by $n^+$ and \eqref{y-eqn0}$_2$ by $n^-$, integrating the resulting equation over $\Omega$, and adding these two equations, we obtain that 
\begin{equation}\label{n^+}
\begin{split}
&\frac12\frac{d}{dt}\int_{\Omega}(|n^+|^2+|n^-|^2)\,dx +\int_{\Omega} (|\nabla n^+|^2+|\nabla n^-|^2) \,dx \\
&=-\int_\Omega \langle\nabla\overline\Psi, [n^+]_+\nabla n^+ - [n^-]_{-}\nabla n^-\rangle\,dx\\
&\leq C\big\|\nabla\overline\Psi\big\|_{L^6(\Omega)} \big\||n^+|+|n^-|\big\|_{L^3(\Omega)} 
\big\||\nabla n^+|+|\nabla n^-|\big\|_{L^2(\Omega)}\\
&\leq C\|\overline\Psi\|_{W^{2,2}(\Omega)} (\|n^+\|_{L^2(\Omega)}+\|n^-\|_{L^2(\Omega)})^\frac12\\
&\cdot
[(\|n^+\|_{L^2(\Omega)}+\|n^-\|_{L^2(\Omega)})+(\|\nabla n^+\|_{L^2(\Omega)}+\|\nabla n^-\|_{L^2(\Omega)})]^\frac12\\
&\cdot[\|\nabla n^+\|_{L^2(\Omega)}+\|\nabla n^-\|_{L^2(\Omega)}]\\
&\leq C[\|{\bar n}^+\|_{L^2(\Omega)}+\|{\bar n}^-\|_{L^2(\Omega)}](\|n^+\|_{L^2(\Omega)}+\|n^-\|_{L^2(\Omega)})^\frac12\\
&\cdot
[(\|n^+\|_{L^2(\Omega)}+\|n^-\|_{L^2(\Omega)})+(\|\nabla n^+\|_{L^2(\Omega)}+\|\nabla n^-\|_{L^2(\Omega)})]^\frac12\\
&\cdot[\|\nabla n^+\|_{L^2(\Omega)}+\|\nabla n^-\|_{L^2(\Omega)}]\\
&\leq \frac12 \big(\|\nabla n^+\|_{L^2(\Omega)}^2+\|\nabla n^-\|_{L^2(\Omega)}^2)\\
&+C\big[1+\big(\|{\bar n}^+\|_{L^2(\Omega)}^4+\|{\bar n}^-\|_{L^2(\Omega)}^4\big)\big]
\cdot\big[\|n^+\|_{L^2(\Omega)}^2+\|n^-\|_{L^2(\Omega)})^2\big].
\end{split}
\end{equation}
This implies that
\begin{equation}\label{n-pm}
\begin{split}
&\frac{d}{dt}\int_{\Omega}(|n^+|^2+|n^-|^2)\,dx +\int_{\Omega} (|\nabla n^+|^2+|\nabla n^-|^2) \,dx \\
&\leq C\big[1+\big(\|{\bar n}^+\|_{L^2(\Omega)}^4+\|{\bar n}^-\|_{L^2(\Omega)}^4\big)\big]
\cdot\big[\|n^+\|_{L^2(\Omega)}^2+\|n^-\|_{L^2(\Omega)})^2\big].
\end{split}
\end{equation}
Applying Gronwall's inequality, we obtain that 
\begin{equation}\label{l2-infty}
\begin{split}
&\sup_{0\le t\le T}\int_{\Omega}(|n^+|^2+|n^-|^2)\,dx+\int_{Q_T}(|\nabla n^+|^2+|\nabla n^-|^2) \,dxdt\\
&\le e^{Ct} \exp\big\{C\int_0^t \big(\|{\bar n}^+\|_{L^2(\Omega)}^4+\|{\bar n}^-\|_{L^2(\Omega)}^4\big)\,d\tau\big\} 
\int_{\Omega}(|n^+_0|^2+|n^-_0|^2)(x)\,dx.
\end{split}
\end{equation}
For $R>0$, if $\bar {\bf y}=(\bar{n}^+,\bar{n}^-)\in Y_T$ belongs to
$$B_R^Y=\big\{\bar{\bf y}:\ \big\|\bar{\bf y}\big\|_{Y_T}\le R\big\},$$
then \eqref{l2-infty} yields that 
$$\big\|F(\bar{\bf y})\big\|_{Y_T}^4=\int_0^T\big(\int_{\Omega}(|n^+|^2+|n^-|^2)(x,t)\,dx)^2\,dt
\le C_0\exp(CT+CR^4) T\le (\frac{R}2)^4,$$
provided that $T=T_1\in (0,1]$ is chosen sufficiently small. Hence there exists a small $T=T_1\in (0,1]$ such that
$F(\bar{\bf y})\in B_{\frac{R}2}^Y\subset B_R^Y.$

Next we want to show that $F: B_R^Y\mapsto B_R^Y$ is a contractive map. 
For $i=1,2$, let ${\bar{\bf y}}_i =\big({\bar n}_i^+, {\bar n}_i^-\big)\in B_R^Y$
and ${\bf y}_i=(n^+_i, n^-_i)=F(\bar{\bf y}_i)\in B_R^Y$ be the solutions of \eqref{psi-eqn} and \eqref{y-eqn}.
Then $n_1^+-n_2^+$ and $n_1^- -n_2^-$ solve 
\begin{equation}\label{psi-eqn12}
-\Delta (\overline{\Psi}_1-\overline{\Psi}_2)
=({\bar n}_1^+-{\bar n}^-_1)-({\bar n}_2^+-{\bar n}^-_2) \ {\rm{in}}\ \Omega; \ \  \frac{\partial (\overline{\Psi}_1-\overline{\Psi}_2)}{\partial\nu}=0 
\ {\rm{on}}\ \partial\Omega,
\end{equation}
\begin{equation}\label{y-eqn}
\begin{cases}
\partial_t (n^+_1-n^+_2) +(w\cdot\nabla) (n^+_1-n^+_2)-\Delta (n^+_1-n^+_2)\\
={\rm{div}}([n^+_1]_+\nabla(\overline{\Psi}_1-\overline{\Psi}_2)) +{\rm{div}}(([n_1^+]_+ - [n_2^+]_+)\nabla\overline{\Psi}_2)& {\rm{in}}\ Q_T,\\
\partial_t (n^-_1-n^-_2) +(w\cdot\nabla) (n^-_1-n^-_2)
-\Delta (n^-_1-n^-_2)\\
=-{\rm{div}}([n^-_1]_+\nabla(\overline{\Psi}_1-\overline{\Psi}_2)) 
-{\rm{div}}(([n_1^-]_+ - [n_2^-]_+)\nabla\overline{\Psi}_2)& {\rm{in}}\ Q_T,\\
(n^+_1-n^+_2, n^-_1-n^-_2)=(0, 0) & {\rm{on}}\ \Omega\times\{t=0\}, \\
\displaystyle\frac{\partial (n^+_1-n^+_2)}{\partial\nu}
= \frac{\partial (n^-_1-n^-_2)}{\partial\nu}=0 & {\rm{on}}\ \partial\Omega\times [0,T].
\end{cases}
\end{equation}

Now multiplying \eqref{y-eqn}$_1$ by $(n^+_1-n^+_2)$, \eqref{y-eqn}$_2$ by $(n^-_1-n^-_2)$, integrating the resulting
equations over $\Omega$, and adding them together, we obtain that
\begin{equation}\label{n12}
\begin{split}
&\frac12\frac{d}{dt}\int_{\Omega}(|n^+_1-n^+_2|^2+|n^-_1-n^-_2|^2)\,dx \\
&+\int_{\Omega} (|\nabla (n^+_1-n^+_2)|^2+|\nabla (n^-_1-n^-_2)|^2) \,dx \\
&=-\int_\Omega [n_1^+]_+\langle\nabla(\overline\Psi_1-\overline\Psi_2), \nabla (n^+_1-n^+_2)\rangle
+([n_1^+]_+ - [n_2^+]_+)\langle\nabla\overline{\Psi}_2, \nabla(n_1^+-n_2^+)\rangle\,dx\\
&\ \ \ +\int_\Omega [n_1^-]_+\langle\nabla(\overline\Psi_1-\overline\Psi_2), \nabla (n^-_1-n^-_2)\rangle
+([n_1^-]_+ - [n_2^-]_+)\langle\nabla\overline{\Psi}_2, \nabla(n_1^- - n_2^-)\rangle\,dx\\
&\leq  C\big\|\nabla\overline{\Psi}_2\big\|_{L^6(\Omega)}
\big(\|[n_1^+]_+-[n_2^+]_+\|_{L^3(\Omega)}\|\nabla(n_1^+-n_2^+)\|_{L^2(\Omega)}\\
&\qquad\qquad\qquad+\|[n_1^-]_+-[n_2^-]_+\|_{L^3(\Omega)}\|\nabla(n_1^--n_2^-)\|_{L^2(\Omega)}\big)\\
&\ +C\big\|\nabla(\overline{\Psi}_1-\overline{\Psi}_2)\big\|_{L^6(\Omega)}
\big(\|[n_1^+]_+\|_{L^3(\Omega)}\|\nabla(n_1^+-n_2^+)\|_{L^2(\Omega)}\\
&\qquad\qquad\qquad\qquad\ \ \ \ \ \ \ +\|[n_1^-]_+\|_{L^3(\Omega)}\|\nabla(n_1^--n_2^-)\|_{L^2(\Omega)}\big)\\
&\leq C\big\|{\bar n}_2^+ - {\bar n}_2^-\big\|_{L^2(\Omega)} \Big\{\|n_1^+-n_2^+\|_{L^2(\Omega)}\|\nabla(n_1^+-n_2^+)\|_{L^2(\Omega)}\\
&\ \ \qquad\qquad\qquad\ \ \ \ \ \ \ \ \ +\|n_1^+-n_2^+\|_{L^2(\Omega)}^\frac12\|\nabla(n_1^+ - n_2^+)\|_{L^2(\Omega)}^\frac32\\
&\qquad\qquad\qquad\ \ \ \ \ \ \ \ \ \ \ +\|n_1^{-}-n_2^{-}\|_{L^2(\Omega)}\|\nabla(n_1^{-}-n_2^{-})\|_{L^2(\Omega)}\\
&\qquad\qquad\qquad\ \ \ \ \ \ \ \ \ \ \ +\|n_1^{-}-n_2^{-}\|_{L^2(\Omega)}^\frac12\|\nabla(n_1^{-}-n_2^{-})\|_{L^2(\Omega)}^\frac32\Big\}\\
&\ +C\Big[\big\|{\bar n}_1^+ - {\bar n}_2^+\big\|_{L^2(\Omega)}
+ \big\|{\bar n}_1^{-} - {\bar n}_2^{-}\big\|_{L^2(\Omega)}\Big]\Big\{\|n_1^+\|_{L^2(\Omega)}\|\nabla(n_1^+-n_2^+)\|_{L^2(\Omega)}\\
&\qquad\qquad\qquad\qquad\qquad\qquad\qquad+\|n_1^+\|_{L^2(\Omega)}^\frac12\|\nabla n_1^+\|_{L^2(\Omega)}^\frac12\|\nabla(n_1^+-n_2^+)\|_{L^2(\Omega)}\\
&\qquad\qquad\qquad\qquad\qquad\qquad\qquad+\|n_1^{-}\|_{L^2(\Omega)}\|\nabla(n_1^{-}-n_2^{-})\|_{L^2(\Omega)}\\
&\qquad\qquad\qquad\qquad\qquad\qquad\qquad+\|n_1^{-}\|_{L^2(\Omega)}^\frac12\|\nabla n_1^{-}\|_{L^2(\Omega)}^\frac12\|\nabla(n_1^{-}-n_2^{-})\|_{L^2(\Omega)}\Big\}\\
&\leq \frac12\Big(\|\nabla(n_1^+-n_2^+)\|_{L^2(\Omega)}^2+\|\nabla(n_1^{-}-n_2^{-})\|_{L^2(\Omega)}^2\Big)\\
&\quad+C\Big(1+\big\|{\bar n}_1^{-}\big\|_{L^2(\Omega)}^4+ \big\|{\bar n}_2^{-}\big\|_{L^2(\Omega)}^4\Big)
\Big( \|n_1^+-n_2^+\|_{L^2(\Omega)}^2+\|n_1^{-}-n_2^{-}\|_{L^2(\Omega)}^2  \Big)\\
&\quad+C\Big\{\big\|n_1^+\big\|_{L^2(\Omega)}\big\|n_1^+\big\|_{H^1(\Omega)}+
\big\|n_1^{-}\big\|_{L^2(\Omega)}\big\|n_1^{-}\big\|_{H^1(\Omega)}\Big\}\\
&\qquad\ \ \ \cdot\Big[\big\|{\bar n}_1^+ - {\bar n}_2^+\big\|_{L^2(\Omega)}^2
+ \big\|{\bar n}_1^{-} - {\bar n}_2^{-}\big\|_{L^2(\Omega)}^2\Big],
\end{split}
\end{equation}
where we have used the following inequalities:  for any $f, g\in H^1(\Omega)$, 
\begin{equation*}
\begin{cases}
\|f_+\|_{L^3(\Omega)}\le \|f\|_{L^3(\Omega)}\le C\|f\|_{L^2(\Omega)}^\frac12\|\nabla f\|_{L^2(\Omega)}^\frac12,\\
\|f_+ -g_+\|_{L^3(\Omega)}\le \|f-g\|_{L^3(\Omega)}\le C\|f-g\|_{L^2(\Omega)}^\frac12\|\nabla (f-g)\|_{L^2(\Omega)}^\frac12.
\end{cases}
\end{equation*}
Therefore we conclude that
\begin{eqnarray}\label{contra1}
&&\frac{d}{dt}\int_{\Omega}(|n^+_1-n^+_2|^2+|n^-_1-n^-_2|^2)\,dx 
+\int_{\Omega} (|\nabla (n^+_1-n^+_2)|^2+|\nabla (n^-_1-n^-_2)|^2) \,dx \nonumber\\
&&\leq C\Big(1+\big\|{\bar n}_1^{-}\big\|_{L^2(\Omega)}^4+ \big\|{\bar n}_2^{-}\big\|_{L^2(\Omega)}^4\Big)
\Big( \|n_1^+-n_2^+\|_{L^2(\Omega)}^2+\|n_1^{-}-n_2^{-}\|_{L^2(\Omega)}^2  \Big)\nonumber\\
&&+C\Big\{\big\|n_1^+\big\|_{L^2(\Omega)}\big\|n_1^+\big\|_{H^1(\Omega)}+
\big\|n_1^{-}\big\|_{L^2(\Omega)}\big\|n_1^{-}\big\|_{H^1(\Omega)}\Big\}\nonumber\\
&&\qquad \cdot\Big[\big\|{\bar n}_1^+ - {\bar n}_2^+\big\|_{L^2(\Omega)}^2
+ \big\|{\bar n}_1^{-} - {\bar n}_2^{-}\big\|_{L^2(\Omega)}^2\Big].
\end{eqnarray}
Applying Gronwall's inequality, we obtain that
\begin{eqnarray}\label{l2-infty1}
&&\sup_{0\le t\le T} \int_{\Omega}(|n^+_1-n^+_2|^2+|n^-_1-n^-_2|^2)\,dx 
+\int_{Q} (|\nabla (n^+_1-n^+_2)|^2+|\nabla (n^-_1-n^-_2)|^2)\,dxdt\nonumber\\
&&\leq \alpha(T) \beta^\frac12(T) \Big\{\int_0^T[\big\|{\bar n}_1^+ - {\bar n}_2^+\big\|_{L^2(\Omega)}^4
+ \big\|{\bar n}_1^{-} - {\bar n}_2^{-}\big\|_{L^2(\Omega)}^4]\,dt\Big\}^\frac12
\end{eqnarray}
where 
$$\alpha(T)=\exp\Big(C\int_0^T\big(1+\big\|{\bar n}_1^{-}\big\|_{L^2(\Omega)}^4+ \big\|{\bar n}_2^{-}\big\|_{L^2(\Omega)}^4\big)\,dt\Big),$$
and
$$
\beta(T)=\Big(\big\||n^+_1|\big\|^2_{L^\infty([0, T], L^2(\Omega))}+\big\||n^{-}_1|\big\|^2_{L^\infty([0, T], L^2(\Omega))}\Big)
\int_0^T(\big\|n_1^+\big\|_{H^1(\Omega)}^2+
\big\|n_1^{-}\big\|_{H^1(\Omega)}^2\big\}\,dt.
$$
It follows from $({\bar n}_1^+, {\bar n}_1^{-})\in B_R^Y$ and \eqref{l2-infty} that for $0<T\le T_1$,
$$\max\big\{\alpha(T), \beta(T)\big\}\le  C(R).$$
Hence \eqref{l2-infty1} yields that for $0<T\le T_1$, 
\begin{eqnarray}\label{l2-infty2}
&&\big\|(n^+_1, n^-_1)-(n^+_2, n^-_2)\big\|_{Y_T}^4\le\int_0^T\big\{\int_{\Omega}(|n^+_1-n^+_2|^2+|n^-_1-n^-_2|^2)\,dx\big\}^2\,dt\nonumber\\
&&\le T \big\{\sup_{0\le t\le T}\int_{\Omega}(|n^+_1-n^+_2|^2+|n^-_1-n^-_2|^2)\,dx\big\}^2\nonumber\\
&&\le T\alpha^2(T)\beta(T)\int_0^T[\big\|{\bar n}_1^+ - {\bar n}_2^+\big\|_{L^2(\Omega)}^4
+ \big\|{\bar n}_1^{-} - {\bar n}_2^{-}\big\|_{L^2(\Omega)}^4]\,dt\nonumber\\
&&\le C(R)T \big\|({\bar n}^+_1, {\bar n}^-_1)-({\bar n}^+_2, {\bar n}^-_2)\big\|_{Y_T}^4\nonumber\\
&&\le 2^{-4}\big\|({\bar n}^+_1, {\bar n}^-_1)-({\bar n}^+_2, {\bar n}^-_2)\big\|_{Y_T}^4,
\end{eqnarray}
provided $T=T_1\le \min\{T_1, \frac{1}{16}C(R)\big\}$. 

This implies that $F: B_R^Y\mapsto B_R^Y$ is a contractive map
with a contraction constant $\frac12$, provided $T_2$ and $R$ are chosen sufficiently small. 
Therefore, there exists a unique fixed point ${\bf y}=(n^+, n^-)\in B_R^Y$ of $F$, i.e., ${\bf y}=F({\bf y})$.
In particular $(n^+, n^-, \Psi)$ is a solution on the interval $[0,T_2]$ of
\begin{equation}\label{psi-eqn2}
-\Delta {\Psi}={n}^+-{n}^- \ {\rm{in}}\ \Omega; \ \  \frac{\partial {\Psi}}{\partial\nu}=0 
\ {\rm{on}}\ \partial\Omega,
\end{equation}
\begin{equation}\label{y-eqn2}
\begin{cases}
\partial_t n^+ +(w\cdot\nabla) n^+-\Delta n^+={\rm{div}}([n^+]_+\nabla{\Psi}) & {\rm{in}}\ Q_T,\\
\partial_t n^- +(w\cdot\nabla) n^--\Delta n^-=-{\rm{div}}([n^-]_+\nabla{\Psi}) & {\rm{in}}\ Q_T,\\
(n^+, n^-)=(n^+_0, n^-_0) & {\rm{on}}\ \Omega\times\{t=0\}, \\
\displaystyle\frac{\partial n^+}{\partial\nu}= \frac{\partial n^-}{\partial\nu}=0 & {\rm{on}}\ \partial\Omega\times [0,T_2],
\end{cases}
\end{equation}
such that $n^\pm\in L^\infty_tL^2_x\cap L^2_tH^1_x(Q_{T_2})$, $\Psi\in L^\infty_tH^2_x\cap L^2_tH^3_x(Q_{T_2})$, and
\begin{equation}\label{l2-infty5}
\begin{split}
\big\|(n^+, n^-)\big\|_{L^\infty_tL^2_x\cap L^2_tH^1_x(Q_{T_2})}
+\big\|\Psi\big\|_{L^\infty_tH^2_x\cap L^2_tH^3_x(Q_{T_2})}\le C\big(\big\|(n_0^+,n_0^-)\big\|_{L^2(\Omega)}, T_2\big). 
\end{split}
\end{equation}
For such a solution $(n^+, n^-, \Psi)$ to \eqref{psi-eqn2} and \eqref{y-eqn2}, let $u\in L^\infty([0,T_2], {\bf H})\cap L^2([0,T_2], {\bf V})$ be a weak solution to the system:
\begin{equation}\label{u-eqn}
\begin{cases}
\partial_t{u}+(w\cdot\nabla){u}-\Delta{u}+\nabla P=-({n}^+-{n}^-)\nabla{\Psi}
&\  {\rm{in}}\ Q_T,\\
{\rm{div}} {u}=0 &\ {\rm{in}}\ Q_T,\\
{u} =u_0 &\ {\rm{in}}\ \Omega\times\{0\},\\
{u}=0 &\ {\rm{on}}\ \partial\Omega\times [0,T_2],
\end{cases}
\end{equation}
Since $(n^+-n^-)\nabla\Psi\in L^\infty([0,T_2], L^\frac32(\Omega))$, it follows from the regularity theory of the
Stokes equation that $\partial_t u, \nabla^2 u\in L^\frac32(Q_{T_2})$, and  $\nabla P\in L^\frac32(Q_{T_2}))$,
and
\begin{equation}\label{u-est0}
\begin{split}
&\big\|u\big\|_{L^\infty([0,T_2], L^2(\Omega))}+\big\|u\big\|_{L^2([0,T_2], H^1(\Omega))}
+\big\|(\partial_t u, \nabla^2 u)\big\|_{L^\frac32(Q_{T_2})}+\big\|\nabla P\big\|_{ L^\frac32(Q_{T_2})}\\
&\le C(\|u_0\|_{L^2(\Omega)}, \|(n_0^+, n_0^-)\|_{L^2(\Omega)}, T_2).
\end{split}
\end{equation} 
From the estimates \eqref{l2-infty5} and \eqref{u-est0}, we can extend $(u, P, n^+, n^-, \Psi)$ beyond
$T_2$ to be a global weak solution of \eqref{A2}-\eqref{IC}-\eqref{BC} on the interval $[0,T]$
such that both \eqref{l2-infty5} and \eqref{u-est0} hold with $T_2$ replaced by $T$. 
Finally, we know that by Lemma \ref{non-negative}, $(u, P, n^+, n^-, \Psi)$ is also a weak solution of the system \eqref{A2}
in $Q_T$.  

It is not hard to verify that since the solution $(u, P, n^+, n^-, \Psi)$ to \eqref{A2} constructed in Step 1 satisfies
the estimates \eqref{l2-infty5} and \eqref{u-est0} (with $T_2=T$), the $L^p$-theory of linear parabolic equations
\cite{LSN} implies that $\partial_t n^+, \partial_t n^-\in L^\frac54(Q_T)$. From
\begin{equation}\label{psi_t}
-\Delta(\partial_t\Psi)=\partial_t n^+-\partial_t n^- \ {\rm{in}}\ \Omega;
\ \frac{\partial}{\partial\nu}(\partial_t\Psi)=0 \ {\rm{on}}\ \partial\Omega,
\end{equation}
we can conclude by the $L^p$-theory of linear elliptic equations that $\nabla^2\partial_t\Psi\in L^\frac54(Q_T)$.

Multiplying the equation \eqref{psi_t} by $\Psi$, \eqref{A2}$_1$ by $u$, integrating  over $\Omega$ and applying integration by parts,
and then adding these two resulting equations together, we can obtain that
\begin{equation}\label{Global-EI1}
\begin{split}
&\int_\Omega (|u|^2+|\nabla\Psi|^2)(x,t)\,dx+2\int_{Q_t} (|\nabla u|^2+|n^+-n^-|^2+(n^++n^-)|\nabla\Psi|^2)\,dxds\\
&=\int_\Omega (|u_0|^2+|\nabla\Psi_0|^2)(x)\,dx+2\int_{Q_t}(n^+-n^-)(w-u)\cdot\nabla\Psi\,dxds
\end{split}
\end{equation}
holds for all $0<t\le T$.

\medskip
\noindent Step 2 {\it Uniqueness}. 
Next we want to prove that there exists at most one  weak solution of \eqref{A2}-\eqref{IC}-\eqref{BC} satisfying
the estimates \eqref{l2-infty5} and \eqref{u-est0}. Let $(u_1,P_1,\Psi_1,n^+_1,n^- _1)$ and $(u_2,P_2,\Psi_2,n^+_2,n^- _2)$ 
be two weak solutions of (\ref{A2}), \eqref{IC}, and \eqref{BC}, satisfying
\eqref{l2-infty5} and \eqref{u-est0}.  Set
$$U=u_1-u_2,\ P=P_1-P_2,\ \Psi=\Psi_1-\Psi_2,
\ N^+=n^+_1-n^+_2,\ N^-=n^-_1-n^-_2.$$
Then 
\begin{equation}\label{A3}\left
\{\begin{array}{l}
\large{ \partial _t U +(w\cdot \nabla ) U -\Delta U + \nabla P = -(N^+-N^-)\nabla\Psi_1-(n_2^+-n_2^-)\nabla\Psi,}\\
\large{\text{div}\ U=0,}\\
\large{ \partial_t N^+ + (w\cdot \nabla )N^+ -\Delta N^+ = \text{div}(N^+ \nabla \Psi_1)+\text{div} (n^+_2 \nabla \Psi),}\\
\large{ \partial_t N^- + (w\cdot \nabla )N^- -\Delta N^- = -\text{div} (N^- \nabla \Psi_1) -\text{div} (n^-_2 \nabla \Psi),}\\
\large{ -\Delta \Psi =N^+ -N^- ,}
\end{array}
\right.
\end{equation}
subject to the initial and boundary condition
\begin{equation}\label{IBC}
\begin{cases}
\displaystyle (U, N^+, N^-)\big |_{t=0}=(0, 0, 0) \ {\rm{on}}\ \Omega,\\
\displaystyle U=0, \frac{\partial N^+}{\partial\nu}=\frac{\partial N^-}{\partial\nu}=\frac{\partial \Psi}{\partial\nu}=0 \ {\rm{on}}\ \partial\Omega\times (0,T).
\end{cases}
\end{equation}
Multiplying (\ref{A3})$_1$ by $U$, (\ref{A3})$_3$ by $N^+$, 
(\ref{A3})$_4$ by $N^-$, and \eqref{A3}$_5$ by $\Psi$, integrating the resulting equations over $\Omega$,
 and adding all these equations together, we obtain that
\begin{equation}\label{est19}
  \begin{split}
     & \frac{1}{2}\frac{d}{dt}\int_{\Omega} (|U|^2+|\nabla \Psi|^2+|N^+|^2+|N^-|^2)\,dx  \\
      & +\int_{\Omega} (|\nabla U|^2 +|N^+ - N^-|^2+ |\nabla N^+ |^2 +|\nabla N^-|^2+(n_2^+ + n_2^- )|\nabla \Psi|^2)\,dx\\
&=-\int_{\Omega}\Big[ (N^+-N^-)U\cdot\nabla\Psi_1 +(n_2^+-n_2^{-})U\cdot\nabla\Psi -(N^+-N^-)w\cdot\nabla\Psi 
\\
&+(N^+-N^-)\nabla\Psi_1\cdot\nabla\Psi
+N^+\nabla\Psi_1\cdot\nabla N^+ +n_2^+\nabla\Psi\cdot\nabla N^+ -n_1^-\nabla\Psi\cdot\nabla N^-\Big]\,dx  \\   
&\leq \frac{1}{2}\big\|N^+-N^-\big\|_{L^2(\Omega)}^2 +C\|U\|_{L^3(\Omega)}^2\|\nabla\Psi_1\|_{L^6(\Omega)}^2
+C\|w\|_{L^\infty(Q)}^2 \big\|\nabla\Psi\big\|_{L^2(\Omega)}^2\\
&\ \ \ +\big\|(n^+_2, n^-_2)\big\|_{L^6(\Omega)}^2\big\|U\big\|_{L^2(\Omega)}^2
+C\big(1+\big\|\nabla\Psi_1\big\|_{L^6(\Omega)}^2\big)\big\|\nabla\Psi\|_{L^3(\Omega)}^2\\
&\ \ \ +C\|\nabla\Psi_1\|_{L^6(\Omega)}^2 \big\|N^+\big\|_{L^3(\Omega)}^2
+\big\|(n_1^-, n_2^+)\big\|_{L^6(\Omega)}^2\|\nabla\Psi\|_{L^3(\Omega)}^2\\
&\ \ \ +\frac{1}{2}\big(\big\|\nabla N^+\big\|_{L^2(\Omega)}^2+\big\|\nabla N^{-}\big\|_{L^2(\Omega)}^2\big).
\end{split}
\end{equation}
By the interpolation inequality, Sobolev's embedding theorem,  and \eqref{l2-infty5}, we have 
\begin{eqnarray*}
\big\|U\big\|_{L^3(\Omega)}^2&\leq& C\big\|U\big\|_{L^2(\Omega)}\big\|\nabla U\big\|_{L^2(\Omega)},\\
\big\|\nabla\Psi_1(t)\big\|_{L^6(\Omega)}&\leq& C\|\Psi_1(t)\|_{H^2(\Omega)}\le C, \ {\rm{a.e.}}\  t\in [0,T],\\
\big\|(n^+_2, n^-_2)\big\|_{L^6(\Omega)}+\big\|n_1^{-}\big\|_{L^6(\Omega)}
&\leq& C\sum_{i=1}^2\big\|(n_i^+, n^{-}_i)\big\|_{H^1(\Omega)},\\
\big\|N^+\big\|_{L^3(\Omega)}^2&\le& C\big\|N^+\big\|_{L^2(\Omega)}^2+
C\big\|N^+\big\|_{L^2(\Omega)}\big\|\nabla N^+\big\|_{L^2(\Omega)},\\
\big\|\nabla\Psi\big\|_{L^3(\Omega)}^2&\le& \big\|\nabla\Psi\big\|_{L^2(\Omega)}
\big\|\nabla\Psi\big\|_{L^6(\Omega)}\\
&\leq&C\big\|\nabla\Psi\big\|_{L^2(\Omega)}\big\|\Psi\big\|_{H^2(\Omega)}\\
&\leq& C\big\|\nabla\Psi\big\|_{L^2(\Omega)}\big\|N^+-N^-\big\|_{L^2(\Omega)}.
\end{eqnarray*}
Putting these estimates into \eqref{est19} and applying Young's inequality, we would obtain 
\begin{equation}\label{est20}
\begin{split}
     & \frac{d}{dt}\int_{\Omega} (|U|^2+|\nabla \Psi|^2+|N^+|^2+|N^-|^2)\,dx  \\
      & +\int_{\Omega} (|\nabla U|^2 +|N^+ - N^-|^2+ |\nabla N^+ |^2 +|\nabla N^-|^2+(n_2^+ + n_2^- )|\nabla \Psi|^2)\,dx\\
      &\le
C\Big(1+\|w\|_{L^\infty(Q)}^2+\big\|(n_1^-, n_2^+)\big\|_{H^1(\Omega)}^2\Big)\\
&\qquad \cdot\Big(\|U\|_{L^2(\Omega)}^2+\big\|N^+\big\|_{L^2(\Omega)}^2+\big\|\nabla\Psi\|_{L^2(\Omega)}^2\Big)\\
&\ +\frac12\Big(\|\nabla N^+\|_{L^2(\Omega)}^2+\|\nabla U\|_{L^2(\Omega)}^2
+\big\|N^+-N^-\big\|_{L^2(\Omega)}^2\Big).
     \end{split}
\end{equation}
This, combined with
$$\gamma(T)=\exp\Big(C\int_0^T (1+\|w|\|_{L^\infty(Q)}^2+\big\|(n_1^-, n_2^+)\big\|_{H^1(\Omega)}^2)\,dt\Big)<\infty,$$
implies that  for any $0<t<T$, 
\begin{equation*}
  \begin{split}
     &  \int_{\Omega} (|U|^2+|\nabla \Psi|^2+|N^+|^2+|N^-|^2)(x,t)\,dx\\
   & \leq    \gamma(T)\int_{\Omega} (|U|^2+|\nabla \Psi|^2+|N^+|^2+|N^-|^2)(x,0)\,dx=0.
  \end{split}
\end{equation*}
%\medskip
%\noindent Step 3 {\it $L^p$-bound}.  If $n^+_0, n_0^-\in L^p(\Omega)$ for some $p\ge 2$, 
%then \eqref{weak_sol2} follows from Lemma \ref{bound}.
This completes the proof.\end{proof}

Next we want to provide a global $L^\frac53$-estimate of the pressure function $P$ of the weak solution
$(u, P, n^+, n^-, \Psi)$ to the system \eqref{A2}. More precisely, we have 
 
\begin{theorem}\label{linear-A1-p,en} Assume $n_0^+, n_0^-\in L^2(\Omega)$ are nonnegative,
$u_0\in {\bf H}$, and $w\in C^\infty(Q_T)$ satisfies $ {\rm{div}}w=0$ in $Q_T$ and $w\cdot\nu=0$ on $\partial\Omega\times [0,T]$.  let $(u,P,n^+,n^-,\Psi)$,
with $\displaystyle\int_{\Omega } P\,dx =0$, be the weak solution of the system (\ref{A2})
in $Q_T$  that satisfies \eqref{weak_sol1}. Then 
$P\in L^{\frac{5}{3}}(Q_T),$ and
\begin{equation}\label{p-bound}
\big\| P\big\|_{L^{\frac{5}{3}}(Q_T)}\le C\Big(1+\|w\|_{L^\infty_tL^2_x\cap L^2_tH^1_x(Q_T)}+\|n_0^+\|_{L^2(\Omega)}+\|n_0^-\|_{L^2(\Omega)}+\|u_0\|_{L^2(\Omega)}\Big).
\end{equation}
Furthermore, for every nonnegative $\phi\in C^{\infty}_0(Q_T)$, it holds that
\begin{equation}\label{suitable1}
  \begin{split}
   2\int_{Q_T} |\nabla u|^2 \phi\,dxdt  = & \int_{Q_T} |u|^2(\partial_t\phi+\Delta \phi)\,dxdt
   +\int_{Q_T} (|u|^2w+2Pu)\cdot \nabla \phi\,dxdt \\
      &  -2\int_{Q_T}\big(\nabla\Psi\otimes\nabla\Psi-\frac12|\nabla\Psi|^2I_3\big): \nabla (u\phi)\,dxdt.
  \end{split}
\end{equation}

\end{theorem}
\begin{proof}
The equation (\ref{A2})$_{1,2}$ can be written as the Stokes equation:
$$
\begin{cases}
\partial _t u  -\Delta u + \nabla P = f,\\
{\rm{div}} u=0,
\end{cases} \ {\rm{in}}\ Q_T,
$$
where $f=-(w\cdot \nabla ) u+\Delta\Psi\nabla\Psi.$
By H\"{o}lder's inequality, we have
\begin{equation}\label{f-bound}
 \begin{split}
 &\big\|f\big\|_{L^\frac53([0,T], L^{\frac{15}{14}}(\Omega))}\\
 &\le C\Big[\big\|w\big\|_{L^{10}_tL^{\frac{30}{13}}_x(Q_T)}\big\|\nabla u\big\|_{L^2(Q_T)}
 +\big\|n^+-n^-\big\|_{L^\infty([0,T], L^2(\Omega))}\big\|\nabla\Psi\big\|_{L^\infty([0,T], L^6(\Omega)}\Big]\\ 
 &\le C\Big(1+\|w\|_{L^\infty_tL^2_x\cap L^2_tH^1_x(Q_T)}+\|n_0^+\|_{L^2(\Omega)}+\|n_0^-\|_{L^2(\Omega)}+\|u_0\|_{L^2(\Omega)}\Big).
 \end{split}
\end{equation}
Here we have used the Sobolev-interpolation inequality:
$$
\big\|w\big\|_{L^{10}_tL^{\frac{30}{13}}_x(Q_T)}\le C\big\|w\big\|_{L^\infty_tL^2_x\cap L^2_tH^1_x(Q_T)}.
$$
In particular, $f\in L^\frac53([0,T], L^{\frac{15}{14}}(\Omega))$.
Applying the theorem by Sohr-Wahl \cite{SW} and \eqref{f-bound}, we obtain that $\nabla P\in L^\frac53([0,T], L^{\frac{15}{14}}(\Omega))$
and 
\begin{equation*}
\begin{split}
&\big\|\nabla P\big\|_{L^{\frac{5}{3}}([0,T], L^{\frac{15}{14}}(\Omega))}\le C\big\|f\big\|_{L^\frac53([0,T], L^{\frac{15}{14}}(\Omega))}\\
&\le C\Big(1+\big\|w\big\|_{L^\infty_tL^2_x\cap L^2_tH^1_x(Q_T)}+\|n_0^+\|_{L^2(\Omega)}+\|n_0^-\|_{L^2(\Omega)}+\|u_0\|_{L^2(\Omega)}\Big).
\end{split}
\end{equation*}
This, combined with Sobolev's inequality, implies that $P\in L^\frac53(Q_T)$ satisfies \eqref{p-bound}.

Mollifying $u, P, f, w\cdot\nabla u$ in $\mathbb{R}^4$,  
we obtain sequences of smooth functions ${u_m}$, ${P_m}$, ${f_m}$,   for $m\in \mathbb N^+$.
Then, for $m$ sufficiently large,
\begin{equation}\label{stokes}
\partial_t u_m -\Delta u_m +\nabla P_m =f_m; \quad \text{div } u_m=0,
\end{equation}
holds in a small neighborhood of $supp\, \phi$. Moreover,
$$u_m \rightarrow u \quad \text{in} \quad L^{3}_{\rm{loc}}(Q_T), \quad \nabla u_m \rightarrow \nabla u \quad \text{in} \quad  L^2_{\rm{loc}}
(Q_T),$$
$$P_m \rightarrow P \quad \text{in} \quad L^{\frac{5}{3}}_{\rm{loc}}(Q_T) , 
\quad f_m \rightarrow f \quad \text{in} \quad  (L^2_tL^\frac32_x)_{\rm{loc}}(Q_T).$$
Multiplying (\ref{stokes}) by $2u_m\phi$ and integrating by parts, we obtain that
\begin{equation*}
  \begin{split}
    2\int_{Q_T}|\nabla u_m|^2 \phi \,dxdt 
=& \int_{Q_T}|u_m|^2(\partial_t\phi +\Delta \phi) \,dxdt + 2\int_{Q_T}P_mu_m\cdot \nabla \phi \,dxdt\\
    & +2\int_{Q_T} f_m\cdot u_m\phi \,dxdt.
 \end{split}
\end{equation*}
Sending $m\rightarrow \infty$, we have 
\begin{equation*}
  \begin{split}
 2\int_{Q_T}|\nabla u|^2 \phi \,dxdt 
=& \int_{Q_T}|u|^2(\partial_t\phi +\Delta \phi) \,dxdt + 2\int_{Q_T}Pu\cdot \nabla \phi \,dxdt\\
    & +2\int_{Q_T} (-w\cdot\nabla u+\Delta\Psi\nabla\Psi)\cdot u\phi \,dxdt.
 \end{split}
\end{equation*}
Note that
$$-2\int_{Q_T} w\cdot\nabla u \cdot u\phi\,dxdt=\int_{Q_T}|u|^2 w\cdot\nabla\phi\,dxdt,$$
and
$$2\int_{Q_T}\Delta\Psi\nabla\Psi \cdot u\phi\,dxdt
=-2\int_{Q_T}\big(\nabla\Psi\otimes\nabla\Psi-\frac12|\nabla\Psi|^2 I_3\big):\nabla(u\phi)\,dxdt.$$
Thus we show that \eqref{suitable1} holds. This completes the proof. \end{proof}

Now recall the well-known Aubin-Lions' compactness Lemma, whose proof can be found at \cite{Temam} section III.
\begin{lemma}\label{compact}
Let $X_0,X_1,X_2$ be three Banach spaces, with $X_0$ and $X_2$ self-reflexive, that satisfy $X_0\subset X_1\subset X_2$. 
Suppose that the embedding of $X_0$ into $X_1$ is compact and the embedding of $X_1$ into $X_2$ is continuous. 
For $p,q\in (0,\infty)$, assume that
$$\{u_k\}_{k\in N} \subset L^{p}([0,T], X_0)$$
is a bounded sequence such that each $u_k $ has a weak derivative $\partial_t u_k$ and the sequence
$$\{\partial _t u_k\}_{k\in N}\in L^q([0,T], X_2)$$
 is also bounded. Then there is a subsequence of $u_k$ converging strongly in $L^p([0,T], X_1)$.
\end{lemma}

Now we utilize Theorem 2.1 to obtain a suitable weak solution to the system \eqref{B}.  For this, we adapt the
``retarded" mollifier technique by Caffarelli-Kohn-Nirenberg \cite{CKN} on the Navier-Stokes equation.

Let $\zeta\in C^\infty_0(\mathbb R^4)$ be non-negative and satisfy 
$$\int_{\mathbb{R}^4} \zeta \,dxdt=1 \ {\rm{and}}\  supp\, \zeta \subset \big\{ (x,t)\in\mathbb R^4:\ |x|^2 <t,\ 1<t<2 \big\}.$$
For $f\in L^1(Q_T)$,  let $\bar{f}=\mathbb{R}^3\times \mathbb{R}\mapsto \mathbb{R}^3$ be
\begin{equation*}
  \bar{f}=\left\{
            \begin{array}{ll}
              f(x,t), & \hbox{if $(x,t)\in \Omega_T$}, \\
              0, & \hbox{otherwise.}
            \end{array}
          \right.
\end{equation*}
Define the ``retarded" mollifier of $f$ by
\begin{equation}\label{modify}
  \Theta_\epsilon (f)(x,t) =\epsilon^{-4} \int_{\mathbb{R}^4} \zeta(\frac{y}{\epsilon}, \frac{\tau}{\epsilon}) \bar{f}(x-y,t-\tau)dyd\tau.
\end{equation}
Then it is well-known (see \cite{CKN} Lemma A.8) that 
\begin{equation*}
\begin{cases}
{\rm{div}} (\Theta_\epsilon(f))=0 \ {\rm{if}}\ {\rm{div}} f=0,\\
\displaystyle\sup_{0\le t\le T} \int_\Omega |\Theta_\epsilon(f)|^2(x,t)\,dx+\int_{Q_T}|\nabla(\Theta_\epsilon(f))|^2\,dxdt\\
\le\displaystyle \sup_{0\le t\le T} \int_\Omega |f|^2(x,t)\,dx+\int_{Q_T}|\nabla f|^2\,dxdt,
\end{cases}
\end{equation*}
and if $f\in L^p(Q_T)$ for $1\le p<\infty$, then $\Theta_\epsilon(f)\to f$ in $L^p(Q_T)$ as $\epsilon\to 0$.
Since $\Theta_\epsilon(f)\cdot\nu$ may not be $0$ on $\partial\Omega\times [0,T]$, we want to modify it as follows.
For $\delta>0$, let $\Omega_\delta$ be the $\delta$-neighborhood of $\Omega$, i.e.
$\Omega_\delta=\big\{y\in\mathbb R^3: {\rm{dist}}(y,\Omega)\le \delta\big\}$, and let
$\Phi_\delta:\Omega\to\Omega_\delta$ be a smooth differeomorphism such that
$$\big\|\Phi_\delta-Id\big\|_{C^1(\Omega)}\le C\delta,$$
where $Id(x)=x$, $x\in\Omega$, is the identity map. 
From the definition, we see that $\Theta_\epsilon(f)=0$ in $(\mathbb R^3\setminus \Omega_{2\epsilon})\times [0,T]$.
Hence $\widetilde{\Theta}_\epsilon(f)(x,t)=\Theta_\epsilon(f(\Phi_{2\epsilon}(x),t))$, $(x,t)\in Q_T$, satisfies that
$\widetilde{\Theta}_\epsilon(f)=0$ on $\partial\Omega\times [0,T]$. If ${\rm{div}}(f)=0$ in $Q_T$, then
$${\rm{div}}\widetilde{\Theta}_\epsilon(f)(x,t)={\rm{tr}}[\nabla\Theta_\epsilon(f)(\Phi_{2\epsilon}(x), t)(\nabla \Phi_{2\epsilon}(x)-I_3)],
\ (x,t)\in Q_T.
$$
Therefore we have that
\begin{equation*}
\begin{split}
\displaystyle\sup_{0\le t\le T} \int_\Omega |\widetilde{\Theta}_\epsilon(f)|^2(x,t)\,dx+\int_{Q_T}|\nabla(\widetilde{\Theta}_\epsilon(f))|^2\,dxdt\\
\le\displaystyle C\big(\sup_{0\le t\le T} \int_\Omega |f|^2(x,t)\,dx+\int_{Q_T}|\nabla f|^2\,dxdt\big),
\end{split}
\end{equation*}
and
\begin{equation*}
\displaystyle \int_{Q_T}|{\rm{div}}(\widetilde{\Theta}_\epsilon(f))|^2\,dxdt
\le\displaystyle C\epsilon^2\int_{Q_T}|\nabla f|^2\,dxdt.
\end{equation*}
For $0<t<T$, let $g_\epsilon(t)\in C^\infty(\overline\Omega)$ satisfy $\displaystyle\int_\Omega g_\epsilon(x,t)\,dx=0$, and solve
\begin{equation*}
-\Delta g_\epsilon(x, t)={\rm{div}}(\widetilde{\Theta}_\epsilon(f))(x,t) \ \ {\rm{in}}\ \ \Omega;\ \ \frac{\partial g_\epsilon}{\partial\nu}(x,t)=0
\ \ {\rm{on}}\ \ \partial\Omega.
\end{equation*}
By the standard elliptic theory, we have that for any $0<t<T$,
\begin{equation*}
\begin{cases}
\displaystyle \int_{\Omega}|\nabla g_\epsilon|^2(x,t)\,dx\le C \int_\Omega |\widetilde{\Theta}_\epsilon(f))|^2(x,t)\,dx\le C\int_\Omega |f|^2(x,t)\,dx,\\
\displaystyle \int_{\Omega} |\nabla^2 g_\epsilon|^2(x,t)\,dx\le C \int_\Omega |{\rm{div}}\widetilde{\Theta}_\epsilon(f))|^2(x,t)\,dx
\le C\epsilon^2\int_\Omega |\nabla f|^2(x,t)\,dx.
\end{cases}
\end{equation*}
Now we define  $\widehat{\Theta}_\epsilon(f)\in C^\infty(\overline\Omega\times (0, T), \mathbb R^3)$ by
letting
$$\widehat{\Theta}_\epsilon(f)(x,t)=\widetilde{\Theta}_\epsilon(f)(x,t)+\nabla g_\epsilon(x,t),\ \ (x,t)\in \overline\Omega\times [0, T].$$ 
Then it is easy to check that for $f\in L^\infty_tL^2_x\cap L^2_tH^1_x(Q_T)$, with ${\rm{div}}(f)=0$ in $Q_T$,
$${\rm{div}}(\widehat{\Theta}_\epsilon(f))=0 \ \ {\rm{in}}\ \ Q_T, \ \ \widehat{\Theta}_\epsilon(f)\cdot\nu=0 \ \ {\rm{on}}\ \ \partial\Omega\times [0,T],$$
\begin{equation*}
\begin{split}
\displaystyle\sup_{0\le t\le T} \int_\Omega |\widehat{\Theta}_\epsilon(f)|^2(x,t)\,dx+\int_{Q_T}|\nabla(\widehat{\Theta}_\epsilon(f))|^2\,dxdt\\
\le\displaystyle C\big(\sup_{0\le t\le T} \int_\Omega |f|^2(x,t)\,dx+\int_{Q_T}|\nabla f|^2\,dxdt\big),
\end{split}
\end{equation*}
and
$$\widehat{\Theta}_\epsilon(f)\to f \  {\rm{in}}\ L^\infty_tL^2_x\cap L^2_tH^1_x(Q_T), \ {\rm{as}}\ \epsilon\to 0.$$

For any large positive integer $M$, set $\epsilon=\frac{T}{M}$.
Let $(u_M,P_M,n^+_M,n^-_M,\Psi_M)$ solve the following system of equations:
\begin{equation}\label{A4}\left
\{\begin{array}{l}
\large{ \partial _t u_M +(\widehat\Theta_\epsilon(u_M)\cdot \nabla ) u_M -\Delta u_M + \nabla P_M = -(n^+_M-n^-_M)\nabla \Psi_M,}\\
\large{\text{div}\ u_M=0,}\\
\large{ \partial_t n^+_M + (\widehat\Theta_\epsilon(u_M)\cdot \nabla )n^+_M -\Delta n^+_M = \text{div } (n^+_M \nabla \Psi_M),}\\
\large{ \partial_t n^-_M + (\widehat\Theta_\epsilon(u_M)\cdot \nabla )n^-_M -\Delta n^-_M = -\text{div } (n^-_M \nabla \Psi_M) ,}\\
\large{ -\Delta \Psi_M =n^+_M -n^-_M.}
\end{array}
\right. \ \ {\rm{in}}\ \  Q_T,
\end{equation}
subject to the initial and boundary condition \eqref{IC} and \eqref{BC}.

Since $\widehat\Theta_\epsilon(u_M)=0$ in $Q_\epsilon$, the system \eqref{A4} decomposes
into the PNP equation and the inhomogeneous Stokes equation, both of which can be solved in the standard ways. 
While in the interval $[\epsilon, 2\epsilon]$,
 $\widehat\Theta_\epsilon (u_M)$ are smooth and their values depend only on the values of $u_M$ and $\Psi_M$ 
 at intervale $[0, \epsilon]$. Hence $(u_M,P_M,n^+_M,n^-_M,\Psi_M)$ of (\ref{A4}) on the interval $\Omega\times [\epsilon, 2\epsilon]$,
 along with the initial condition $(u_M, n^+_M, n^-_M)(\cdot, \epsilon)$ and the boundary condition \eqref{BC}, 
can be solved by Theorem \ref{linear-A1}. Keeping this process in each interval $(m\epsilon,(m+1)\epsilon), 0\leq m\leq M-1$,
we obtain a global solution $(u_M,P_M,n^+_M,n^-_M,\Psi_M)$ to \eqref{A4}, \eqref{IC}, and \eqref{BC}.

It follows from Lemma 2.1, Proposition 2.1 (for $p=2$), \eqref{l2-infty5} and \eqref{u-est0} of Theorem 2.1, and 
\eqref{p-bound} of Theorem 2.2 that $\{u_M\}$ is bounded in $L^\infty_tL^2_x\cap L^2_tH^1_x(Q_T)$,
$\{n^\pm_M\}$ are non-negative, and bounded in $L^\infty_tL^2_x\cap L^2_tH^1_x(Q_T)$, $\Psi_M$ is bounded in $L^\infty_tH^2_x\cap L^2_tH^3_x(Q_T)$,
and $\{P_M\}$ is bounded in $L^\frac53(Q_T)$. 

By the equations $(\ref{A4})_1, (\ref{A4})_3, (\ref{A4})_4$, we have that
\begin{equation*}
\begin{cases}
\partial _t u_M =-\text{div }( u_M \otimes\widehat\Theta_\epsilon(u_M)-\nabla u_M + P_MI_3)  -(n^+_M-n^-_M)\nabla \Psi_M,\\
 \partial_t n^+_M= -\text{div }( n^+_M \widehat\Theta_\epsilon(u_M) -\nabla n^+_M -n^+_M \nabla \Psi_M),\\
\partial_t n^-_M =- \text{div}(n^-_M \widehat\Theta_\epsilon(u_M) -\nabla n^-_M + n^-_M \nabla \Psi_M).
\end{cases}
\end{equation*}
It is straightforward to see that $\{\partial _t u_M\}_{M\in N}, \{\partial_t n^+_M\}_{M\in N},\{\partial_t n^-_M\}_{M\in N} $ 
are bounded in the space
$$L^{\frac{5}{3}} ([0,T], W^{-1,\frac{5}{2}}(\Omega)).$$
Hence we can apply  Lemma \ref{compact} with
$$
\left\{
  \begin{array}{ll}
    X_0 := H^1 (\Omega), & \\
    X_1 := L^2 (\Omega),& \\
    X_2 := W^{-1,\frac{5}{2}},
  \end{array}
\right.$$
to conclude that there exist $u\in L^\infty_tL^2_x\cap L^2_tH^1_x(Q_T)$,
$n^\pm\in L^\infty_tL^2_x\cap L^2_tH^1_x(Q_T)$, $\Psi\in L^\infty_tH^2_x\cap L^2_tH^3_x(Q_T)$,
and  $P\in L^\frac53(Q_T)$ such that as $M\to\infty$, after passing to a subsequence,
\begin{equation}\label{c-1}
u_m\rightharpoonup u \ {\rm{in}}\ L^2_tH^1_x(Q_T),
\   u_M \rightarrow u\ \text{in} \  L^q(Q_T) \ \forall 1<q<\frac{10}3,
\end{equation}
\begin{equation}\label{c-2}
\begin{cases}
  (n^+_M,\ n^-_M) \rightharpoonup (n^+,\ n^-) \ {\rm{in}}\ L^2_tH^1_x(Q_T),\\
   (n^+_M,\ n^-_M) \rightarrow (n^+,\ n^-)  \ \text{in  }  L^l(\Omega_T) \ \forall 1<l<\frac{10}3, 
   \end{cases}
\end{equation}
\begin{equation}\label{c-3}
\nabla\Psi_M \rightarrow \nabla\Psi  \quad \text{in  }  L^4(Q_T),
\end{equation}
and
\begin{equation}\label{c-4}
P_M \rightharpoonup P \quad \text{in  }  L^\frac53(Q_T).
\end{equation}
With \eqref{c-1}, \eqref{c-2}, \eqref{c-3}, and \eqref{c-4}, we can easily verify that $(u, P, n^+, n^-, \Psi)$ is
a weak solution of \eqref{B}, \eqref{IC}, and \eqref{BC}. 

Since $(u_M, n^+_M, n^-_M, \Psi_M)$ satisfies the global energy equality \eqref{Global-EI1}, 
with $(u, n^+, n^-, \Psi)$ and $w$ replaced by $(u_M, n^+_M, n^-_M, \Psi_M)$ and $\widehat{\Theta}_\epsilon(u_M)$
respectively, and since
$$n^+_M\to n^+,  \ n^-_M\to n^-,
\ \widehat\Theta_\epsilon(u_M)\to u, \ \nabla\Psi_M\to \nabla\Psi \ \  {\rm{in}}\ \ L^3(Q_T),
$$
it is not hard to verify that as $\epsilon\to 0$, 
$$2\int_{Q_t} (n^+_M-n^-_M)(\widehat\Theta_\epsilon(u_M)-u)\cdot\nabla\Psi_M\,dxds\rightarrow 0, \ \forall 0<t\le T, $$
and hence for any $0<t<T$, 
\begin{equation*}
\begin{split}
&\int_\Omega (|u|^2+|\nabla\Psi|^2)\,dx+2\int_{Q_t} (|\nabla u|^2+|n^+-n^-|^2+(n^++n^-)|\nabla\Psi|^2)\,dxds\\
&\le\liminf_{\epsilon\to 0} \big\{
\int_\Omega (|u_M|^2+|\nabla\Psi_M|^2)\,dx+2\int_{Q_t} (|\nabla u_M|^2+|n^+_M-n^-_M|^2+(n^+_M+n^-_M)|\nabla\Psi_M|^2)\,dxds\big\}\\
&=\liminf_{\epsilon\to 0} \big(\int_\Omega (|u_0|^2+|\nabla\Psi_0|^2)\,dx+2\int_{Q_t}(n^+_M-n^-_M)(\widehat{\Theta}_\epsilon(u_M)-u_M)\cdot\nabla\Psi_M\,dxds\big)\\
&=\int_\Omega (|u_0|^2+|\nabla\Psi_0|^2)\,dx,
\end{split}
\end{equation*}
which yields that  $(u, n^+, n^-, \Psi)$ satisfies the global energy inequality \eqref{Global-EI}.

Finally we need to verify that $(u, P, n^+, n^-,\Psi)$ satisfies the local energy inequality \eqref{GE-1}.
For this, consider a test function $\phi\in C^{\infty}(\overline{Q_T}) $ with $\phi\geq 0$ and $supp\, \phi \Subset Q_T$. 
By Theorem \ref{linear-A1-p,en}, we have
\begin{equation}\label{local}
  \begin{split}
   &2\int_{Q_T} |\nabla u_M|^2 \phi\,dxdt  =  \int_{Q_T} |u_M|^2(\partial_t\phi+\Delta \phi)\,dxdt\\
   &+\int_{Q_T} (|u_M|^2\widehat\Theta_\epsilon(u_M)+2P_Mu_M)\cdot \nabla \phi\,dxdt \\
   &-2\int_{Q_T}\big(\nabla\Psi_M\otimes\nabla \Psi_M-\frac12|\nabla\Psi_M|^2I_3\big):\nabla(u_M\phi)\,dxdt.
  \end{split}
\end{equation}
As $M\rightarrow \infty$, by the lower semicontinuity we have that
$$2\int_{Q_T} |\nabla u|^2\phi\,dxdt\le\liminf_{M\to\infty} \int_{Q_T} |\nabla u_M|^2 \phi\,dxdt,$$
while by (\ref{c-1})--(\ref{c-4}) and $\widehat\Theta_\epsilon(u_M)\to u$ in $L^3(Q_T)$ as $\epsilon\to 0$, we have
\begin{equation*}
\begin{split}
& \int_{Q_T} |u_M|^2(\partial_t\phi+\Delta \phi)\,dxdt+\int_{Q_T} (|u_M|^2\widehat\Theta_\epsilon(u_M)+2P_Mu_M)\cdot \nabla \phi \,dxdt\\
& -2\int_{Q_T}\big(\nabla\Psi_M\otimes\nabla \Psi_M-\frac12|\nabla\Psi_M|^2I_3\big):\nabla(u_M\phi)\,dxdt\\
 &\to 
 \int_{Q_T} |u|^2(\partial_t\phi+\Delta \phi)\,dxdt+\int_{Q_T} (|u|^2u+2Pu)\cdot \nabla \phi \,dxdt\\
 &\qquad-2\int_{Q_T}\big(\nabla\Psi\otimes\nabla \Psi-\frac12|\nabla\Psi|^2I_3\big):\nabla(u\phi)\,dxdt.
\end{split} 
\end{equation*}
Hence \eqref{GE-1} follows.

\section{the $\epsilon$-regularity, part I}

In this section, we will prove the partial regularity of suitable weak solutions to \eqref{B}. The crucial steps are 
two levels of $\epsilon$-regularities. 

For $(x,t)\in Q_T$ and $r>0$, set
$$B_r(x)=\big\{y\in\mathbb R^3: \ |y-x|<r\big\},\ \  Q_r(x,t)=\big\{(y,\tau)\ | \ |y-x|<r, \ t-r^2<\tau <t \big\},$$
and denote $B_r(0)$ and $Q_r(0,0)$ by $B_r$ and $Q_r$.

\begin{lemma}\label{u-small} There exist $\epsilon_0>0$ and $\theta_0\in (0, \frac12)$ such that if $(u,P, n^+, n^-, \Psi)$ 
is a suitable weak solution of the system (\ref{B}) in $Q_T$, which satisfies, for a $(x_0, t_0)\in Q_T$ and 
$0<r_0<\min\big\{{\rm{dist}}(x_0,\partial \Omega),  \sqrt{t_0}\big\}$,
\begin{equation}\label{small-1}
  r_0^{-2}\int_{Q_{r_0}(x_0,t_0)}|u|^3\,dxdt +\big(r_0^{-1}\int_{Q_{r_0}(x_0,t_0)}|\nabla \Psi|^4 \,dxdt\big)^\frac34
  +\big(r_0^{-2}\int_{Q_{r_0}(x_0,t_0)}|P|^{\frac{3}{2}}\,dxdt\big)^2 < \epsilon_0^3,
\end{equation}
then
\begin{equation}\label{decay-1}
\begin{split}
&(\theta_0r_0)^{-2}\int_{Q_{\theta_0r_0}(x_0,t_0)}|u|^3\,dxdt
+\big((\theta_0r_0)^{-2}\int_{Q_{\theta_0r_0}(x_0,t_0)}|P|^{\frac{3}{2}}\,dxdt\big)^2 \\
& \le \frac12\Big[r_0^{-2}\int_{Q_{r_0}(x_0,t_0)}|u|^3\,dxdt + \big(r_0^{-1}\int_{Q_{r_0}(x_0,t_0)}|\nabla \Psi|^4 \,dxdt\big)^\frac34\\
 &\qquad +\big(r_0^{-2}\int_{Q_{r_0}(x_0,t_0)}|P|^{\frac{3}{2}}\,dxdt\big)^2\Big].
\end{split}
\end{equation}
\end{lemma}
\begin{proof} For $z_0=(x_0,t_0)\in Q_T$ and $r_0>0$, define the scaling functions
$$\big(\tilde{u}, \tilde{P}, \tilde{n}^+, \tilde{n}^-, \tilde{\Psi}\big)(x,t)
=\big(r_0u, r_0^2P, n^+, n^-,\Psi\big)(x_0+r_0x, t_0+r_0^2 t).$$
We can verify that if $(u,P, n^+, n^-, \Psi)$ solves \eqref{B}, then $\big(\tilde{u}, \tilde{P}, \tilde{n}^+, \tilde{n}^-, \tilde{\Psi}\big)$
solves the following system: 
\begin{equation}\label{BB}\left
\{\begin{array}{l}
\large{ \partial _t \tilde{u} +(\tilde{u}\cdot \nabla ) \tilde{u} -\Delta \tilde{u} + \nabla \tilde{P} = -r_0^2(\tilde{n}^+-\tilde{n}^-)\nabla \tilde{\Psi},}\\
\large{\text{div}\ \tilde{u}=0,}\\
\large{ \partial_t \tilde{n}^+ + (\tilde{u}\cdot \nabla )\tilde{n}^+ -\Delta \tilde{n}^+ = \text{div} (\tilde{n}^+ \nabla \tilde\Psi),}\\
\large{ \partial_t \tilde{n}^- + (\tilde{u}\cdot \nabla )\tilde{n}^- -\Delta \tilde{n}^- = -\text{div} (\tilde{n}^- \nabla \tilde\Psi),}\\
\large{ -\Delta \tilde\Psi =r_0^2(\tilde{n}^+ -\tilde{n}^-),}
\end{array}
\right.
\end{equation}
From \eqref{BB}$_5$, we can see that
$$-r_0^2(\tilde{n}^+-\tilde{n}^-)\nabla \tilde{\Psi}=\Delta\tilde{\Psi}\cdot\nabla\tilde{\Psi}
={\rm{div}}(\nabla\tilde{\Psi}\otimes\nabla\tilde{\Psi}-\frac12|\nabla\tilde\Psi|^2I_3).$$
Thus \eqref{BB}$_1$ can be rewritten as 
\begin{equation}\label{BB1}
\partial _t \tilde{u} +(\tilde{u}\cdot \nabla ) \tilde{u} -\Delta \tilde{u} + \nabla \tilde{P}=
{\rm{div}}(\nabla\tilde{\Psi}\otimes\nabla\tilde{\Psi}-\frac12|\nabla\tilde\Psi|^2I_3).
\end{equation}
Because of the invariance of the first four equations of \eqref{BB} under translations and scalings, we will assume $z_0=(0,0)$ and
$r_0=1$.  We prove \eqref{decay-1} by contradication. Suppose the conclusion were false. Then for any $\theta\in (0,\frac12)$,
there would exist a sequence of suitable weak solutions $(u_i, P^i, n_i^+, n_i^-, \Psi_i)$ of \eqref{B} in $Q_1$ such that
\begin{equation}\label{small-2}
\int_{Q_{1}}|u_i|^3\,dxdt +\big(\int_{Q_1} |\nabla \Psi_i|^4\,dxdt\big)^\frac34
  +\big(\int_{Q_1}|P_i|^{\frac{3}{2}}\,dxdt\big)^2 = \epsilon_i^3\rightarrow 0,
\end{equation}
and
\begin{equation}\label{decay-2}
\begin{split}
&\theta^{-2}\int_{Q_{\theta}}|u_i|^3\,dxdt
+\big(\theta^{-2}\int_{Q_{\theta}}|P_i|^{\frac{3}{2}}\,dxdt\big)^2 \\
&>\frac12\Big[\int_{Q_{1}}|u_i|^3\,dxdt + \big(\int_{Q_1}|\nabla \Psi_i|^4) \,dxdt\big)^\frac34
  +\big(\int_{Q_{1}}|P_i|^{\frac{3}{2}}\,dxdt\big)^2\Big].
\end{split}
\end{equation}
Now we define the blowing up sequences $v_i=\frac{u_i}{\epsilon_i}, \ R_i=\frac{P_i}{\epsilon_i}, \ \Phi_i=\frac{\Psi_i}{\epsilon_i}$ on
$Q_1$. Then $(v_i, R_i)$ solves the system
\begin{equation}\label{BB2}
\begin{cases}
\qquad\qquad\partial_t v_i +\epsilon_i v_i\cdot\nabla v_i-\Delta v_i +\nabla R_i&= \epsilon_i {\rm{div}}\big(\nabla\Phi_i\otimes \nabla\Phi_i
-\frac12|\nabla\Phi_i|^2I_3\big),\\
\qquad\qquad\qquad\qquad\qquad\qquad\qquad {\rm{div}} v_i &=0, 
\end{cases}
\end{equation}
and satisfies
\begin{equation}\label{small-3}
\int_{Q_{1}}|v_i|^3\,dxdt + \big(\int_{Q_1}|\nabla \Phi_i|^4 \,dxdt\big)^\frac34
  +\big(\int_{Q_1}|R_i|^{\frac{3}{2}}\,dxdt\big)^2 = 1,
\end{equation}
\begin{equation}\label{decay-3}
\begin{split}
\theta^{-2}\int_{Q_{\theta}}|v_i|^3\,dxdt
+\big(\theta^{-2}\int_{Q_{\theta}}|R_i|^{\frac{3}{2}}\,dxdt\big)^2 
>\frac12.
\end{split}
\end{equation}
Moreover, since $(u_i, P_i, \Psi_i)$ satisfies the local energy inequality \eqref{GE-1}, we can see that
$(v_i, R_i, \Phi_i)$ satisfies a rescaled version of \eqref{GE-1}: for any $0\le \phi\in C_0^\infty(Q_1)$,
\begin{equation}\label{GE-20}
 \begin{split}
     & 2\int_{Q_1}|\nabla v_i|^2 \phi \,dxdt \\
&\leq  \int_{Q_1}|v_i|^2(\phi_t +\Delta \phi) \,dxdt +\int_{Q_1}(\epsilon_i|v_i|^2 +2 R_i)v_i\cdot \nabla \phi \,dxdt\\
    & -2\int_{Q_1} \epsilon_i (\nabla\Phi_i\otimes\nabla\Phi_i-\frac12|\nabla\Phi_i|^2 I_3):\nabla(v_i\phi) \,dxdt\\
&\leq   \int_{Q_1}|v_i|^2(\phi_t +\Delta \phi) \,dxdt +\int_{Q_1}(\epsilon_i|v_i|^2 +2 R_i)v_i\cdot \nabla \phi \,dxdt\\
&\ \ +C\epsilon_i\int_{Q_1} |\nabla\Phi_i|^2 (\phi+|v_i||\nabla\phi|)\,dxdt +\int_{Q_1}|\nabla v_i|^2\phi\,dxdt.
 \end{split}
\end{equation}
By choosing suitable test functions $\phi$, \eqref{GE-20} and \eqref{small-3} imply
that $v_i\in L^\infty_tL^2_x\cap L^2_tH^1_x(Q_\frac12)$
and there exists $C>0$ such that
\begin{equation}\label{V-bound}
\sup_{i\ge 1}\big\|v_i\big\|_{L^\infty_tL^2_x\cap L^2_tH^1_x(Q_\frac12)}\le C.
\end{equation}
 Moreover, we see from \eqref{BB2}  that 
\begin{equation}\label{Vt-bound}
\big\|\partial _t v_i\big\|_{L^\frac32([-1,0], W^{-1,\frac32}(B_1))}\le C.
\end{equation}
Indeed, for $\phi\in L^3([-1,0], W^{1,3}_0(B_1))$, we have
\begin{eqnarray*}
&& \big|\int_{Q_1} \partial_t v_i \phi\,dxdt\big|\\
&&=\big|\int_{Q_1} [(\epsilon_i v_i\otimes v_i-\nabla v_i) : \nabla \phi +R_i \text{div}\, \phi
-\epsilon_i (\nabla\Phi_i\otimes\nabla\Phi_i-\frac12|\nabla\Phi_i|^2 I_3):\nabla\phi]\,dxdt \big|\\
&&\le C \big(\|v_i\|_{L^3(Q_1)}^2 +\|\nabla \Phi_i\|_{L^3(Q_1)}^2+\|R_i\|_{L^{\frac32}(Q_1)}\big)\|\nabla \phi\|_{L^3(Q_1)}\\
&&\leq C \big\|\phi\big\|_{L^3([-1,0], W^{1,3}_0(B_1))}.
\end{eqnarray*}
From \eqref{V-bound} and \eqref{Vt-bound}, we can apply Lemma \ref{compact} to conclude that
after passing to a subsequence, there exist $v\in L^\infty_tL^2_x\cap L^2_tH^1_x(Q_\frac12)$,
$R\in L^\frac53(Q_\frac12)$ and $\Phi\in L^4_tW^{1,4}_x(Q_\frac12)$ such that
\begin{equation}\label{strong-conv1}
v_i\rightharpoonup v \ {\rm{in}}\ L^2_tH^1_x(Q_\frac12), 
\ v_i\rightarrow v \ {\rm{in}}\ L^3(Q_\frac12),
\end{equation}
and
\begin{equation}\label{weak-conv1}
R_i\rightharpoonup R \ {\rm{in}}\ L^\frac53(Q_\frac12),\  \Phi_i\rightharpoonup \Phi \ {\rm{in}}\ L^4_tW^{1,4}_x(Q_\frac12).
\end{equation}
Passing to the limit in \eqref{BB2}, we see that $(v, R)$ solves the Stokes equation:
\begin{equation}\label{limit-stokes}
\partial_t v-\Delta v+\nabla R=0; \ {\rm{div}} v=0\ \ {\rm{in}}\ \ Q_\frac12.
\end{equation}
Therefore by the standard theory on the Stokes equation, we can conclude that $v\in C^\infty(Q_\frac12)$, and
for any $\theta\in (0, \frac12)$,
\begin{equation}\label{decay-4}
\theta^{-2}\int_{Q_\theta}|v|^3\,dxdt\le C\theta^3 \int_{Q_\frac12}|v|^3\,dxdt\le C\theta^3.
\end{equation}
This and \eqref{strong-conv1} imply that for $i$ sufficiently large,
\begin{equation}\label{decay-5}
\theta^{-2}\int_{Q_\theta}|v_i|^3\,dxdt\le C\theta^3+o(1).
\end{equation}
Here $o(1)$ denotes a quantity such that $\displaystyle\lim_i o(1)=0$. 

As for the pressure function $R_i$, taking divergence of \eqref{BB2}$_1$ yields that $R_i$ solves
the Poisson equation:
\begin{equation}\label{BB3}
\Delta R_i= \epsilon_i {\rm{div}}^2\big(\nabla\Phi_i\otimes \nabla\Phi_i
-\frac12|\nabla\Phi_i|^2I_3-v_i\otimes v_i\big)\ \ {\rm{in}}\ \ B_\frac12.
\end{equation} 
By the Calderon-Zygmund theory, we can show that
\begin{equation}\label{decay-6}
\begin{split}
\theta^{-2}\int_{Q_\theta}|R_i|^\frac32\,dxdt&\le C\theta^{-2}\epsilon_i^\frac32 \int_{Q_1} (|v_i|^3+|\nabla \Phi_i|^3)\,dxdt
+C\theta^3 \int_{Q_1} |R_i|^\frac32\,dxdt\\ 
&\le C\theta^{-2}\epsilon_i^\frac32+C\theta^3.
\end{split}
\end{equation}
Adding \eqref{decay-5} and \eqref{decay-6} together, we obtain that
\begin{equation}\label{decay-7}
\theta^{-2}\int_{Q_\theta}|v_i|^3\,dxdt+\big(\theta^{-2}\int_{Q_\theta}|R_i|^\frac32\,dxdt\big)^2
\le C\theta^3+C\theta^{-2}\epsilon_i^\frac32+o(1)\le \frac14,
\end{equation}
provided we choose a sufficiently small $\theta\in (0,\frac12)$ and a sufficiently large $i$. 
It is clear that \eqref{decay-7} contradicts to \eqref{decay-3}. The proof is complete.
\end{proof}

Keep iterating Lemma \ref{u-small}, we obtain the following decay property.

\begin{corollary}\label{u-small2} There exist $\epsilon_0>0$ and $\theta_0\in (0, \frac12)$ such that if $(u,P, n^+, n^-, \Psi)$ 
is a suitable weak solution of the system (\ref{B}) in $Q_T$, which satisfies, for a $z_0=(x_0, t_0)\in Q_T$, 
$0<r_0<\min\big\{{\rm{dist}}(x_0,\partial \Omega),  \sqrt{t_0}\big\}$, and $0<\alpha<4$
\begin{equation}
\label{small-8}
\begin{split}
  &\max\Big\{r_0^{-2}\int_{Q_{r_0}(z_0)}|u|^3\,dxdt
  +\big(r_0^{-2}\int_{Q_{r_0}(z_0)}|P|^{\frac{3}{2}}\,dxdt\big)^2, \\
  &\qquad \ \sup_{0<r\le r_0} \big(r^{-(1+\alpha)}\int_{Q_r(z_0)}|\nabla\Psi|^4\,dxdt\big)^\frac34\Big\}< \frac12\epsilon_0^3,
  \end{split}
\end{equation}
then for any positive integer $k\in \mathbb N^+$, 
\begin{equation}\label{decay-8}
\begin{split}
(\theta_0^kr_0)^{-2}\int_{Q_{\theta_0^kr_0}(z_0)}|u|^3\,dxdt
+\big((\theta_0^kr_0)^{-2}\int_{Q_{\theta_0^kr_0}(z_0)}|P|^{\frac{3}{2}}\,dxdt\big)^2 \le C\epsilon_0^3\big(\frac12\big)^k.
\end{split}
\end{equation}
\end{corollary} 
\begin{proof} It is readily seen that \eqref{decay-8} follows from Lemma \ref{u-small} for $k=1$. Note that
\eqref{small-8} and \eqref{decay-8} for $k=1$ yield that
\begin{equation*}
\begin{split}
&(\theta_0r_0)^{-2}\int_{Q_{\theta_0r_0}(z_0)}|u|^3\,dxd+\big((\theta_0 r_0)^{-1}\int_{Q_{\theta_0 r_0}(z_0)}|\nabla\Psi|^4)\,dxdt\big)^\frac34\\
&\quad+\big((\theta_0r_0)^{-2}\int_{Q_{\theta_0r_0}(z_0)}|P|^{\frac{3}{2}}\,dxdt\big)^2<\epsilon_0^3.
\end{split}
\end{equation*}
Hence applying Lemma \ref{u-small}, we obtain that
\begin{eqnarray*}
&&(\theta_0^2r_0)^{-2}\int_{Q_{\theta_0^2r_0}(z_0)}|u|^3\,dxdt
+\big((\theta_0^2r_0)^{-2}\int_{Q_{\theta_0^2r_0}(z_0)}|P|^{\frac{3}{2}}\,dxdt\big)^2\\
&&\le \frac12\Big[(\theta_0r_0)^{-2}\int_{Q_{\theta_0r_0}(z_0)}|u|^3\,dxdt+\big((\theta_0r_0)^{-1}\int_{Q_{\theta_0r_0}(z_0)}|\nabla\Psi|^4\,dxdt\big)^\frac34\\
&&\ \ \ \ \ \ +\big((\theta_0r_0)^{-2}\int_{Q_{\theta_0r_0}(z_0)}|P|^{\frac{3}{2}}\,dxdt\big)^2\Big]\\
&&\le \frac12\Big[(\theta_0r_0)^{-2}\int_{Q_{\theta_0r_0}(z_0)}|u|^3\,dxdt
+\big((\theta_0r_0)^{-2}\int_{Q_{\theta_0r_0}(z_0)}|P|^{\frac{3}{2}}\,dxdt\big)^2\\
&&\ \ \ \ \ \ \ \ +\big((\theta_0 r_0)^{-1}\int_{Q_{\theta_0 r_0}(z_0)}|\nabla\Psi|^4\,dxdt\big)^\frac34\Big]\\
&&\le \frac12\Big[\frac12\Big(r_0^{-2}\int_{Q_{r_0}(z_0)}|u|^3\,dxdt +\big(r_0^{-1}\int_{Q_{r_0}(z_0)}|\nabla\Psi|^4\,dxdt\big)^\frac34\\
&&\qquad+\big(r_0^{-2}\int_{Q_{r_0}(z_0)}|P|^{\frac{3}{2}}\,dxdt\big)^2\Big)
+\big((\theta_0 r_0)^{-1}\int_{Q_{\theta_0 r_0}(z_0)}|\nabla\Psi|^4\,dxdt\big)^\frac34\Big]\\
&&\le (\frac12)^2\Big[r_0^{-2}\int_{Q_{r_0}(z_0)}|u|^3\,dxdt
+\big(r_0^{-2}\int_{Q_{r_0}(z_0)}|P|^{\frac{3}{2}}\,dxdt\big)^2\Big]\\
&&\ \ \ +(\frac12)^2\big(r_0^{-1}\int_{Q_{r_0}(z_0)}|\nabla\Psi|^4\,dxdt\big)^\frac34
+\frac12\big((\theta_0 r_0)^{-1}\int_{Q_{\theta_0 r_0}(z_0)}|\nabla\Psi|^4\,dxdt\big)^\frac34\\
&&\le (\frac12)^2\Big[r_0^{-2}\int_{Q_{r_0}(z_0)}|u|^3\,dxdt
+\big(r_0^{-2}\int_{Q_{r_0}(z_0)}|P|^{\frac{3}{2}}\,dxdt\big)^2\Big]\\
&&\ \ \ +(\frac 12)^2\epsilon_0^3 r_0^\alpha \big[\theta_0^\alpha +\frac12\big].
\end{eqnarray*}
Hence we have that for $k\ge 1$,
\begin{eqnarray*}
&&(\theta_0^kr_0)^{-2}\int_{Q_{\theta_0^kr_0}(z_0)}|u|^3\,dxdt
+\big((\theta_0^kr_0)^{-2}\int_{Q_{\theta_0^kr_0}(z_0)}|P|^{\frac{3}{2}}\,dxdt\big)^2\\
&&\le (\frac12)^k\Big[r_0^{-2}\int_{Q_{r_0}(z_0)}|u|^3\,dxdt
+\big(r_0^{-2}\int_{Q_{r_0}(z_0)}|P|^{\frac{3}{2}}\,dxdt\big)^2\Big]\\
&&\ \ \ +(\frac 12)^2\epsilon_0^3 r_0^\alpha \Big[\theta_0^{\alpha(k-1)} +\frac12\theta_0^{\alpha(k-2)}+\cdots+
(\frac12)^{k-2}\theta_0^\alpha\Big]\\
&&\le (\frac12)^k\Big[r_0^{-2}\int_{Q_{r_0}(z_0)}|u|^3\,dxdt
+\big(r_0^{-2}\int_{Q_{r_0}(z_0)}|P|^{\frac{3}{2}}\,dxdt\big)^2\Big]+2^{-(k-1)}(\theta_0 r_0)^\alpha\epsilon_0^3\\
&&\le C\epsilon_0^3 2^{-k}.
\end{eqnarray*}
This yields \eqref{decay-8} and completes the proof. \end{proof} 

With \eqref{decay-8}, we can now prove the following $\epsilon_0$-regularity property.
\begin{theorem}\label{l3-small-cont} There exists $\epsilon_0>0$ such that
for any $0<T\le \infty$, $u_0\in{\bf H}$, and $0\le n_0^\pm \in L^2(\Omega)$
with $\int_\Omega n_0^+\,dx= \int_\Omega n_0^-\,dx$, if $(u, P, n^+, n^-,\Psi)$ is
the suitable weak solution obtained by Theorem 1.3 (i), which satisfies
\begin{equation}\label{small-9}
r_0^{-2}\int_{Q_{r_0}(z_0)}|u|^3\,dxdt+\big(r_0^{-1}\int_{Q_{r_0}(z_0)}|\nabla\Psi|^4\,dxdt\big)^\frac34+ \big(r_0^{-2}\int_{Q_{r_0}(z_0)}|P|^\frac32\,dxdt\big)^2\le\epsilon_0^3,
\end{equation}
for $z_0=(x_0,t_0)\in \Omega\times (0,\infty)$ and $0<r_0<\min\big\{{\rm{dist}}(x_0,\partial\Omega), \sqrt{t_0}\big\}$,
then $(u, n^+, n^-,\Psi)\in C^\infty(Q_{\frac{r_0}2}(z_0))$.
\end{theorem}
\begin{proof} It follows from \eqref{a-bound} and Sobolev's embedding theorem
that $\nabla\Psi\in L^\infty_t L^6_x(Q_T)$, and
\begin{equation} \label{16-bound}
\big\|\nabla\Psi\big\|_{L^\infty_t L^6_x(Q_T)}\le C\big\|\Psi\big\|_{L^\infty_tH^2_x(Q_T)}
\le C\big(\|u_0\|_{L^2(\Omega)}, \|(n_0^+, n_0^-)\|_{L^2(\Omega)}\big).
\end{equation}
This implies that
\begin{equation}\label{morrey_bd}
\int_{Q_r(z)} |\nabla\Psi|^4\,dxdt \le Cr^{3} \big\|\nabla\Psi\big\|_{L^\infty_t L^6_x(Q_T)}^4\le Cr^{3}, \ \forall Q_r(z)\subset Q_T.
\end{equation}
It follows from \eqref{morrey_bd} and \eqref{small-9} that  for any $\alpha\in (0, 2)$,
the condition \eqref{small-8} holds on $Q_{\frac{r_0}2}(z_1)$ for any $z_1\in Q_{\frac{r_0}2}(z_0)$, provided
we may choose a smaller $r_0>0$, depending on $\epsilon_0$. Thus by Corollary \eqref{u-small2}, we conclude that
there exists $\theta_0\in (0,\frac12)$ such that
\begin{equation}\label{decay-9}
\begin{split}
(\theta_0^kr_0)^{-2}\int_{Q_{\theta_0^kr_0}(z_1)}|u|^3\,dxdt
+\big((\theta_0^kr_0)^{-2}\int_{Q_{\theta_0^kr_0}(z_1)}|P|^{\frac{3}{2}}\,dxdt\big)^2 \le C\epsilon_0^3\big(\frac12\big)^k,
\end{split}
\end{equation}
for any $z_1\in Q_{\frac{r_0}2}(z_0)$. Therefore there exists $\tau_0\in (0,1)$ such that
\begin{equation}\label{decay-10}
s^{-2}\int_{Q_s(z_1)} |u|^3\,dxdt +\big(s^{-2}\int_{Q_s(z_1)} |P|^\frac32\,dxdt\big)^2\le Cs^{3\tau_0}, 
\end{equation}
for all $z_1\in Q_{\frac{r_0}2}(z_0)$ and $0<s<\frac{r_0}2$. From \eqref{decay-10}, we can repeat the same argument  
of Lemma 3.1 and Corollary 3.1 to improve the exponent $\tau_0$ such that \eqref{decay-10} remains to
be true for all $\tau_0\in (0,1)$.

Now we plan to apply the Riesz potential estimates between parabolic Morrey spaces to show that $u\in L^q(Q_{\frac{r_0}2}(z_0)$
for any $1<q<\infty$, analogous to
that by Huang-Wang \cite{HW}, Hineman-Wang \cite{HLW1}, and Huang-Lin-Wang \cite{HLW}.  

For any open set $U\subset\mathbb R^3\times \mathbb R$, $1\le p<\infty$, and $0\le\lambda\le 5$, define
the Morrey space $M^{p,\lambda}(U)$ by
$$
M^{p,\lambda}(U)
:=\Big\{ f\in L^p_{\rm{loc}}(U):\ \big\|f\big\|_{M^{p,\lambda}(U)}^p 
=\displaystyle\sup_{z\in U, r>0} r^{\lambda-5}\int_{Q_r(z)} |f|^p\,dxdt<\infty\Big\}.
$$
It follows from \eqref{morrey_bd} and \eqref{decay-10} that for any $\alpha\in (0,1)$, it holds that
$$(u, \nabla \Psi)\in M^{3, 3(1-\alpha)}\big(Q_{\frac{r_0}2}(z_0)\big), 
\ P\in M^{\frac32, 3(1-\alpha)}\big(Q_{\frac{r_0}2}(z_0)\big).
$$
We now proceed with the estimation of $u$. Let $\eta\in C_0^\infty(\mathbb R^4)$ be a cut-off function of
$Q_{\frac{r_0}2}(z_0)$ such that $0\le\eta\le 1$, $\eta\equiv 1$ in $Q_{\frac{z_0}2}(z_0)$,
and $|\partial_t\eta|+|\nabla^2\eta|\le Cr_0^{-2}$. Let $v:\mathbb R^3\times (0,\infty)\mapsto \mathbb R^3$ solve the Stokes equation:
\begin{equation}\label{stokes}
\begin{cases}
\partial_t v -\Delta v +\nabla P= - {\rm{div}}\big[\eta^2\big(u\otimes u +(\nabla \Psi\otimes\nabla \Psi-\frac12|\nabla \Psi|^2I_3)\big)\big]
& {\rm{in}}\ \mathbb R^4_+,\\
{\rm{div}} v=0 & {\rm{in}}\ \mathbb R^4_+,\\
v(\cdot, 0)=0 & {\rm{in}}\ \mathbb R^3.
\end{cases}
\end{equation}
By using the Oseen kernel (see Leray \cite{Leray}), an estimate of $v$ can be given by 
\begin{equation}\label{duhamel2}
|v(x,t)|\le C\mathcal{I}_1(|X|)(x,t), \ \forall (x,t)\in\mathbb R^3\times (0,\infty),
\end{equation}
where 
$$X=\eta^2\big[u \otimes u +(\nabla \Psi\otimes\nabla \Psi-\frac12|\nabla \Psi|^2I_3)\big],
$$
and $\mathcal{I}_1$ is the Reisz potential of order $1$ on $\mathbb R^4$ defined by
$$\mathcal{I}_1(g)(x,t)=\int_{\mathbb R^4} \frac{|g(y,s)|}{\delta^4((x,t), (y,s))}\,dyds, \ \forall g\in L^1(\mathbb R^4).$$
We can verify that $X\in M^{\frac32, 3(1-\alpha)}(\mathbb R^4)$ and
\begin{equation*}
\begin{split}
\big\|X\big\|_{M^{\frac32, 3(1-\alpha)}(\mathbb R^4)}&\le C\Big[\|u\|_{M^{3, 3(1-\alpha)}(Q_{\frac{r_0}2}(z_0))}^2+
\|\nabla \Psi\|_{M^{3, 3(1-\alpha)}(Q_{\frac{r_0}2}(z_0))}^2\Big]\\
&\le C(1+\epsilon_0).
\end{split}
\end{equation*}
Hence we conclude that  $v\in M^{\frac{3(1-\alpha)}{1-2\alpha}, 3(1-\alpha)}(\mathbb R^4)$
and 
\begin{equation}\label{morrey6}
\Big\|v\Big\|_{M^{\frac{3(1-\alpha)}{1-2\alpha}, 3(1-\alpha)}(\mathbb R^4)} 
\le C \Big\|X\Big\|_{M^{\frac32, 3(1-\alpha)}(\mathbb R^4)}
\le C(1+\epsilon_0).
\end{equation}
By taking $\alpha\uparrow \frac12$, we conclude that for any $1<q<\infty$, 
$v\in L^q(Q_{\frac{r_0}2}(z_0))$ and
\begin{equation}\label{lpestimate3}
\big\|v\big\|_{L^q(Q_{\frac{r_0}2}(z_0))}\le C(q, r_0, \epsilon_0). 
\end{equation}
Note that $u-v$ solves the linear homogeneous Stokes equation:
$$\partial_t(u-v)-\Delta(u-v)+\nabla P=0, \ {\rm{div}}(u-v)=0 \ \ {\rm{in}}\ \ Q_{\frac{r_0}2}(z_0).$$
Then $u-v\in L^\infty(Q_{\frac{r_0}4}(z_0))$. Therefore for any $1<q<\infty$,
$u\in L^q(Q_{\frac{r_0}4}(z_0))$ and
\begin{equation}\label{lpestimate4}
\big\|u\big\|_{L^q(Q_{\frac{r_0}4}(z_0))}\le C(q, r_0, \epsilon_0).
\end{equation}
From $\Psi\in L^\infty_tH^2_x\cap L^2_tH^3_x(Q_T)$ and the Sobolev inequality, we have that $\Delta\Psi\in 
L^{\frac{10}3}(Q_T)$, $\nabla\Psi \in L^q(Q_T)$ for $q>5$, and
$$\big\|\Delta\Psi\big\|_{L^{\frac{10}3}(Q_T)}+\big\|\nabla\Psi \big\|_{L^q(Q_T)}
\le C\big\| \Psi\big\|_{L^\infty_tH^2_x\cap L^2_tH^3_x(Q_T)}\le C.$$
Since $n^+$ solves
$$\partial_t n^+ -\Delta n^+ =(\Delta\Psi) n^+ - (u-\nabla\Psi)\cdot\nabla n^+\ \ {\rm{in}}\ \ Q_{\frac{r_0}4}(z_0),
$$
where $(u-\nabla\Psi)\in L^q(Q_T)$ and $\Delta\Psi \in L^{\frac{q}2}(Q_T)$ for some $q>5$, we can
apply the standard theory of linear parabolic equation \cite{LSN} to conclude that there exists $\beta\in (0,1)$
such that  $n^+\in C^\beta(Q_{\frac{r_0}4}(z_0))$, and
\begin{equation}\label{n+}
\big\|n^+\big\|_{C^\beta(Q_{\frac{r_0}4}(z_0))}\le C(r_0, \epsilon_0).
\end{equation}
Similarly, we can show that 
$n^-\in C^\beta(Q_{\frac{r_0}4}(z_0))$, and
\begin{equation}\label{n-}
\big\|n^-\big\|_{C^\beta(Q_{\frac{r_0}4}(z_0))}\le C(r_0, \epsilon_0).
\end{equation}
Substituting the estimates \eqref{n+} and \eqref{n-} into the equation \eqref{B}$_5$ for $\Psi$, we conclude that
$\nabla^2\Psi\in L^\infty([t_0-\frac{r_0^2}{64}, t_0], C^\alpha(B_{\frac{r_0}8}(x_0))$ and
\begin{equation}\label{psi}
\big\|\nabla^2\Psi\big\|_{L^\infty([t_0-\frac{r_0^2}{64}, t_0], C^\alpha(B_{\frac{r_0}8}(x_0))}\le C(r_0, \epsilon_0).
\end{equation}
Substituting \eqref{n+}, \eqref{n-}, and \eqref{psi} into the equation \eqref{B}$_{1,2}$, we conclude that
$u\in C^\beta(Q_{\frac{r_0}{16}}(z_0))$ and
\begin{equation}\label{u-holder}
\big\|u\big\|_{C^\beta(Q_{\frac{r_0}{16}}(z_0))}\le C(r_0,\epsilon_0).
\end{equation}
By a bootstrap argument, we can eventually show that $(u, n^+, n^-,\Psi)\in C^\infty(Q_{\frac{r_0}{32}}(z_0))$.
\end{proof}

\begin{remark} Similar to \cite{Sche} and \cite{CKN}, Theorem \ref{l3-small-cont} yields that
$(u, n^+, n^-, \Psi)$ is smooth away from a closed set $\Sigma$, with $\mathcal{P}^{\frac53}(\Sigma)=0$.
\end{remark}

\section{the $\epsilon$-regularity, part II}

In this section, we will improve the size estimate of the singular set $\Sigma$ for suitable weak solutions $(u, P, n^+, n^-,\Psi)$
obtained by Theorem \ref{main}. The argument is based on the A-B-C-D Lemmas, originally due to \cite{CKN}.  
Namely, we want to establish the following theorem.

\begin{theorem}\label{grade u-small} Under the same assumptions as in Theorem \ref{main},  
there exists $\epsilon_1>0$ such that if $(u,P,n^+,n^-, \Psi)$ is the suitable weak solutions of (\ref{B}) given by
Theorem \ref{main}, and satisfies 
\begin{equation}\label{small-10}
 \limsup_{r\rightarrow 0}\frac{1}{r} \int_{Q_r(z_0)}|\nabla u|^2  \,dxdt < \epsilon_1^2
\end{equation}
for $z_0=(x_0,t_0)\in Q_T$,  then $(u, n^+, n^-, \Psi)$ is smooth near $z_0$.
\end{theorem}

For simplicity, we will assume $z_0=(0,0)\in Q_T$. 
In order to prove Theorem \ref{grade u-small}, we first recall the following interpolation inequality, see \cite{CKN}. 

\begin{lemma}\label{eq-Lq}
For $u\in H^1(\mathbb{R}^3)$,
$$\int_{B_r}|u|^q\,dx \leq C \big( \int_{B_r} |\nabla u|^2 \,dx\big)^{\frac{q}2-a} \big( \int_{B_r} |u|^2 \,dx\big)^{a} 
+Cr^{3(1-\frac{q}2)}\big( \int _{B_r} |u|^2\,dx \big)^{\frac{q}{2}},$$
for any $B_r\subset\mathbb R^3$,  $2\leq q\leq 6\quad \text{and} \quad a=\frac{3}{2}(1-\frac{q}6).$
\end{lemma}

Assume $z_0=(0,0)$.  Set
$$A(r)=\sup_{-r^2\le t\le 0} r^{-1} \int_{B_r\times\{t\}} |u|^2 \,dx,$$
$$B(r)=r^{-1}\int_{Q_r} |\nabla u|^2 \,dxdt,$$
$$C(r)=r^{-2}\int_{Q_r} |u|^3 \,dxdt ,$$
$$D(r)=r^{-2}\int_{Q_r} |P|^{\frac{3}{2}} \,dxdt.$$

By Lemma \ref{eq-Lq}, we see that for any $0<r\le\rho$,  it holds that 
\begin{equation}\label{Cr}
  C(r)\leq C_0 \Big [ \big( \frac{r}{\rho} \big)^3 A^\frac32(\rho)
  +\big( \frac{\rho}{r}\big)^3 A^\frac34(\rho) B^\frac34(\rho)\Big].
\end{equation}

Now we need to estimate the pressure function.

\begin{lemma}\label{D}
Let $(u,P,n^+,n^-, \Psi)$ be a suitable weak solution of (\ref{B}) in $Q_1$ given by Theorem
\ref{main}. Then for any $0<r\le\frac{\rho}{2}$, we have
\begin{equation}\label{Dr}
  D(r) \leq C \left[ \frac{r}{\rho} D(\rho) +\left( \frac{\rho}{r} \right)^2A^{\frac{3}{4}}(\rho)B^{\frac{3}{4}}(\rho)+\big(\frac{\rho}{r}\big)^2\rho^\frac32 \right].
\end{equation}

\end{lemma}

\begin{proof}  Taking divergence of \eqref{B}$_1$, we obtain
\begin{equation}\label{P-eqn}
\begin{split}
-\Delta P={\rm{div}}^2\big[(u-(u)_\rho)\otimes (u-(u)_\rho)+(\nabla \Psi \otimes \nabla \Psi-\frac12|\nabla\Psi|^2I_3)\big] \ \ {\rm{in}}\ \ B_\rho.
\end{split}
\end{equation}
Here $(u)_\rho$  denotes the average of $u$  over $B_\rho$.

Let $\eta\in C_0^\infty(\mathbb R^3)$ be a cut off function of $B_{\frac{\rho}2}$ such that 
  \begin{equation}
    \left\{
      \begin{array}{ll}
	\eta=1, &\text{ in }B_{\frac{\rho}2}, \\
	\eta=0, & \text{ outside}\ B_{\rho}, \\
	0\leq \eta\leq 1, & |\nabla \eta|\leq 8\rho^{-1}. 
      \end{array}
      \right.
      \label{}
    \end{equation}
    Define an auxiliary function
    \begin{align*}
      P_1(x, t)&=-\int_{\mathbb R^3}\nabla_y^2 G(x-y):\eta^2(y)\big[ (u-({u})_\rho)\otimes ({u}-({u})_\rho)\\
      &\qquad\qquad\qquad+\big(\nabla \Psi\otimes \nabla \Psi-\frac12|\nabla\Psi|^2 I_3\big)\big](y, t)\,dy,
    \end{align*}
    Then we have
    $$-\Delta P_1={\rm{div}}^2\big[(u-(u)_\rho)\otimes (u-(u)_\rho)+(\nabla \Psi\otimes \nabla \Psi-\frac12|\nabla\Psi|^2I_3)\big]\ \text{ in }\ B_{\frac{\rho}2},$$
    and
    $$-\Delta (P-P_1)=0 \ \ \text{ in }\ \ B_{\frac{\rho}2}.$$
    For $P_1$, we apply the Calderon-Zygmund theory to deduce
    \begin{eqnarray}
      \int_{\mathbb R^3} |P_1|^{\frac{3}{2}}\,dx&\le& C\int_{\mathbb R^3} \eta^3(|{u}-({u})_\rho|^3+|\nabla\Psi|^3)\,dx\nonumber\\
     &\le& C\int_{B_\rho}(|{u}-({u})_\rho|^3+|\nabla \Psi|^3)\,dx
      \label{}
    \end{eqnarray}
Since $P-P_1$ is harmonic in $B_{\frac{\rho}2}$, we get that for $0<r<\frac{\rho}2$, 
\begin{equation*}
\begin{split}
\frac{1}{r^2}\int_{B_r} |P-P_1|^{\frac{3}{2}}\,dx&\leq C (\frac{r}{\rho}) \frac{1}{\rho^2} \int_{B_{\frac{\rho}2}} |P-P_1|^{\frac{3}{2}}\,dx\\
&\le C (\frac{r}{\rho})\big[\frac{1}{\rho^2} \int_{B_{\frac{\rho}2}} |P|^{\frac{3}{2}}\,dx
+\frac{1}{\rho^2} \int_{B_{\frac{\rho}2}} |P_1|^{\frac{3}{2}}\,dx\big].
\end{split}
\end{equation*}
Integrating it over $[-r^2, 0]$,  we can show that
 \begin{eqnarray*}
 &&\frac{1}{r^2}\int_{Q_r}|P|^{\frac{3}{2}}\,dxdt\\
 &&\le C\big(\frac{r}{\rho}\big)\frac{1}{\rho^2}\int_{Q_\rho}|P|^{\frac{3}{2}}\,dxdt+
 C\big(\frac{\rho}{r}\big)^2 \frac{1}{\rho^2}\int_{Q_\rho}(|{u}-({u})_\rho|^3+|\nabla \Psi|^3)\,dxdt\\
 &&\le C\big(\frac{r}{\rho}\big)\frac{1}{\rho^2}\int_{Q_\rho}|P|^{\frac{3}{2}}\,dxdt+
 C\big(\frac{\rho}{r}\big)^2 \frac{1}{\rho^2}\int_{Q_\rho}|{u}-({u})_\rho|^3\,dxdt+C\big(\frac{\rho}{r}\big)^2\rho^\frac32,
 \end{eqnarray*}
 where we have used the inequality \eqref{morrey_bd} in the last step.
 
 This, combined with the interpolation inequality
 \begin{eqnarray*}
  && \frac{1}{\rho^2}\int_{Q_\rho}|{u}-({u})_\rho|^3\,dxdt\nonumber\\
  &&{\le C \sup_{-\rho^2\leq t\leq 0}\big(\frac{1}{\rho} \int_{B_\rho}|{u}|^2\,dx \big)^{\frac{3}{4}}
   \cdot \big(\frac{1}{\rho} \int_{Q_\rho}|\nabla {u}|^2\,dxdt\big)^{\frac{3}{4}}, }
   \label{}
 \end{eqnarray*}
implies that
 $$D(r)\leq C\Big[ (\frac{r}{\rho}) D(\rho)+(\frac{\rho}{r})^2 A^{\frac{3}{4}}(\rho)B^{\frac{3}{4}}(\rho)+\big(\frac{\rho}{r}\big)^2\rho^\frac32\Big].$$
This completes the proof. \end{proof}

\begin{proof}[Proof of Theorem \ref{grade u-small}] Here we follow the presentation by \cite{DHW} closely.
For $0<\theta<\frac12$ and $0<\rho<1$, let $0\le\phi\in C_0^\infty(Q_{\theta\rho})$ be such that 
$$
\phi=1 \ \ {\rm{in}}\ \ Q_{\frac{\theta\rho}2},\  |\nabla\phi|\le \frac{4}{\theta\rho},\  \ |\nabla^2\phi|+|\partial_t\phi|\le \frac{16}{(\theta\rho)^2}.
$$
Applying the local energy inequality \eqref{GE-1} and using ${\rm{div}} u=0$, we obtain
\begin{equation}\label{GE-2}
   \begin{split}
     &\sup_{-(\theta\rho)^2\le t\le 0}\int_{\Omega}|u|^2\phi^2\, dx + 2\int _{\Omega\times [-(\theta\rho)^2, 0]}|\nabla u|^2 \phi^2 \,dxdt \\
&\leq  \int_{\Omega\times [-(\theta\rho)^2, 0]}|u|^2(|\partial_t\phi|+|\nabla\phi|^2+|\nabla^2\phi|) \,dxdt \\
&\ \ +  \int_{\Omega\times [-(\theta\rho)^2, 0]}\big(||u|^2-(|{u}|^2)_\rho| +2|P|)|u| |\nabla\phi|\, dxdt\\
& \ \ +2\int_{\Omega\times [-(\theta\rho)^2, 0]} \big|\nabla\Psi\otimes\nabla\Psi-\frac12|\nabla\Psi|^2 I_3\big|(|\nabla u|\phi+|u||\nabla\phi|)\, dxdt,
 \end{split}
\end{equation}
where 
$$(|{u}|^2)_\rho=\fint_{B_\rho} |u|^2 \,dx$$
is the average of $|u|^2$ over $B_\rho$.
By using Sobolev's inequality, we have
$$\big(\int_{B_{\rho}}||u|^2- (|{u}|^2)_\rho^2|^{\frac{3}{2}} \,dx \big)^{\frac{2}{3}} \leq  C \int_{B_\rho} |u||\nabla u| \,dx.$$
By H\"older's inequality, we can bound
\begin{equation*}
\begin{split}
&\int_{\Omega\times [-(\theta\rho)^2, 0]} \big|\nabla\Psi\otimes\nabla\Psi-\frac12|\nabla\Psi|^2 I_3\big|(|\nabla u|\phi+|u||\nabla\phi|)\, dxdt\\
&\le c\int_{Q_{\theta\rho}} |\nabla\Psi|^2|\nabla u|\,dxdt+ c(\theta\rho)^{-1}\int_{Q_{\theta\rho}} |\nabla\Psi|^2|u|\,dxdt\\
&\le c(\theta\rho)^\frac12 B^\frac12(\theta\rho) \big(\int_{Q_{\theta\rho}}|\nabla\Psi|^4\,dxdt\big)^\frac12 \\
&\ \ +c(\theta\rho)^{-1}\big(\int_{Q_{\theta\rho}}|\nabla\Psi|^3\,dxdt\big)^\frac23 \big(\int_{Q_{\theta\rho}}|u|^3\,dxdt\big)^\frac13\\
&\le c(\theta\rho)^2\big(B^\frac12(\theta\rho)+C^\frac13(\theta\rho)\big),
\end{split}
\end{equation*}
where we have used in the last step \eqref{morrey_bd} and
$$\int_{Q_{\theta\rho}}|\nabla\Psi|^4\,dxdt\le c(\theta\rho)^\frac32.$$
Substituting these two estimates into \eqref{GE-2}, we obtain
\begin{equation}\label{A,B}
  \begin{split}
A(\frac12\theta\rho)+B(\frac12\theta\rho)&\leq  c\Big[C^{\frac{2}{3}}(\theta\rho) +A^\frac12(\theta\rho)B^\frac12(\theta\rho)C^\frac13(\theta\rho)\\
&\ \ \ +C^\frac13(\theta\rho)D^\frac23(\theta\rho)+(\theta\rho)^2B^\frac12(\theta\rho)+(\theta\rho)^2 C^\frac13(\theta\rho)\Big]\\
&\le c\Big[ C^{\frac{2}{3}}(\theta\rho)+A(\theta\rho)B(\theta\rho)+(\theta\rho)^4+(\theta\rho)^2B^\frac12(\theta\rho)+D^\frac43(\theta\rho)\Big].
  \end{split}
\end{equation}
Thus we obtain 
$$A^\frac32(\frac12\theta\rho)\le c\Big[C(\theta\rho)+A^\frac32(\theta\rho)B^\frac32(\theta\rho)+D^2(\theta\rho)+(\theta\rho)^6+(\theta\rho)^3B^\frac34(\theta\rho)\Big].
$$
While we also have
$$D^2(\theta\rho)\le c\theta^2\big[D^2(\rho)+\theta^{-6}A^\frac32(\rho)B^\frac32(\rho)+\theta^{-6}\rho^3\big],$$
$$C(\theta\rho)\le c\Big[\theta^3 A^\frac32(\rho)+\theta^{-3}A^\frac34(\rho)B^\frac34(\rho),$$
and
$$A^\frac32(\theta\rho)B^\frac32(\theta\rho)
\le \theta^{-3} A^\frac32(\rho)B^\frac32(\rho).$$
Putting all these estimates together, we arrive at
\begin{equation*}
\begin{split}
&A^\frac32(\frac12\theta\rho)+D^2(\frac12\theta\rho)\\
&\le c\Big[\theta^2(D^2(\rho)+A^\frac32(\rho))+\theta^{-8}A^\frac32(\rho)B^\frac32(\rho) +\theta^2+\theta^{-4}\rho^3+\theta^6\rho^6+\theta^\frac94\rho^3B^\frac34(\rho)\Big]\\
&\le c\big(\theta^2+\theta^{-8}B^\frac32(\rho)\big)\big(A^\frac32(\rho)+D^2(\rho)\big)+c\big(\theta^2+\theta^{-4}\rho^3+\theta^6\rho^6+\theta^\frac94\rho^3B^\frac34(\rho)\big).
\end{split}
\end{equation*}
For $\epsilon_1>0$ given by Theorem , let $\theta_0\in (0,\frac12)$ be such that
$$c\theta_0^2=\min\{\frac14, \frac18\epsilon_1^2\}.$$
Since
$$\limsup_{r\to 0} \frac{1}{r}\int_{Q_r} |\nabla u|^2\,dxdt<\epsilon_1^2,$$
we can choose $\rho_0>0$ such that
$$c\theta_0^{-2}B^\frac32(\rho)\le \frac14, \ \forall 0<\rho<\rho_0,$$
and
$$
c\big(\theta_0^2+\theta_0^{-4}\rho^3+\theta_0^6\rho^6+\theta_0^\frac94\rho^3B^\frac34(\rho)\big)\le \frac12\epsilon_1^2, \ \forall 0<\rho<\rho_0.$$
Therefore we obtain that there exist $\theta_0\in (0,\frac12)$ and $\rho_0>0$ such that
$$A^\frac32(\frac12\theta_0 \rho) + D^2 (\frac12\theta_0 \rho) \leq \frac12 \big(A^\frac32( \rho) + D^2 (\rho)^2\big) +\frac12\epsilon_1^2, \ \forall 0<\rho<\rho_0.$$
Iterating this inequality yields that
 \begin{align}\label{decay1}
 A^\frac32((\frac12{\theta_0)^k \rho})+D^2((\frac12{\theta_0)^k \rho})\le \frac{1}{2^k} (A^\frac32(\rho)+D^2(\rho))
 +\epsilon_1^2
 \end{align} 
holds for all $0<\rho<\rho_0$ and $k\ge 1$.

Employing \eqref{decay1} and \eqref{Cr}, we obtain that
\begin{align}\label{decay2}
C((\frac12\theta_0)^{k}\rho)&\le c\big[(\frac12\theta_0)^3 A^\frac32((\frac12\theta_0)^{k-1}\rho)
+(\frac12\theta_0)^{-3} A^\frac34((\frac12\theta_0)^{k-1}\rho) B^\frac34((\frac12\theta_0)^{k-1}\rho)\big]\nonumber\\
&\le c\big[(\frac12\theta_0)^3 +(\frac12\theta_0)^{-3} \varepsilon_1^\frac32\big]
\big[\frac{1}{2^{k-1}} (A^\frac32(\rho)+D^2(\rho))
 +\epsilon_1^2\big]
\end{align}
holds for all $0<\rho<\rho_0$ and $k\ge 1$.

Putting \eqref{decay1} and \eqref{decay2} together, we obtain that
\begin{align}\label{decay3}
\limsup_{k\to\infty} \big[C((\frac12\theta_0)^{k}\rho)+D^2((\frac12\theta_0)^{k}\rho)\big]
&\le c\big[1+(\frac12\theta_0)^3 +(\frac12\theta_0)^{-3} \epsilon_1^\frac32]\epsilon_1^2
\le \frac12\epsilon_0^3,
\end{align}
holds for all $\rho\in (0,\rho_0)$, provided $\epsilon_1=\epsilon_1(\theta_0, \epsilon_0)>0$ is chosen
sufficiently small. Therefore, by Theorem \ref{l3-small-cont}
$({u}, n^+, n^-, \Psi)$ is smooth near $z_0=(0,0)$.  This completes the proof. \end{proof}

\medskip
\noindent{\it Completion of Proof of Theorem \ref{main}}:  Define the singular set of $(u, n^+, n^-, \Psi)$ by 
$$\Sigma=\Big\{(x,t)\in Q_T\ |\ \limsup_{r\rightarrow 0} r^{-1}\int_{Q_r(x,t)}|\nabla u|^2\,dxdt>\epsilon_1^2\Big\}.$$
From Theorem \ref{grade u-small}, we know that $\Sigma$ is closed and $(u, n^+, n^-, \Psi)\in C^\infty(Q_T\setminus\Sigma)$.

Let $U$ be a small neighborhood of $\Sigma$ and let $\delta >0$. For each $(x,t)\in \Sigma$,  choose $0<r<\delta$ such that
$$r^{-1}\int_{Q_r(x,t)} |\nabla u|^2\,dxdt >\epsilon_1^2 \text{  and  } Q_{r}(x,t) \subset U.$$
By Vitali's five time covering Lemma, there exists a disjoint subfamily $\{\,Q_{r_i}(x_i,t_i)\,\}$ such that
$$\Sigma \subset \bigcup_i \,Q_{5r_i}(x_i,t_i).$$
Hence
\begin{eqnarray*}
\mathcal{P}^1_{5\delta}(\Sigma)\le \sum_i 5r_i \leq 5\epsilon_1^{-2} \sum_{i}\int_{Q_{r_i}(x_i,t_i)}|\nabla u|^2\,dxdt\leq 5\epsilon^{-2}_1 \int_{U}|\nabla u|^2\,dxdt.
\end{eqnarray*}
Sending $\delta\to 0$, this implies that $\mathcal{P}^1(\Sigma)=0$. The proof is now complete.
\end{proof}

\bigskip
\noindent{\bf Acknowledgements}.  The first author is partially supported by Natural Science Foundation of China (grant
No.11601342, No.61872429, No.11871345, ).
The second author is partially supported by NSF grant 1764417. Both the first and third
authors would like to express their gratitudes to Department of Mathematics, Purdue University, 
where the project was initiated during their visit to the second author. The authors wish to thank Qiao Liu for some helpful discussion.

%\bigskip

\end{document}